\documentclass[a4paper,twoside]{article}      
\usepackage{amsmath,amssymb,amsfonts,amsthm,amscd,mathtools}
\usepackage{graphics}                 
\usepackage{color}                    
\usepackage{hyperref}                
\usepackage{ mathrsfs }
\usepackage{indentfirst}
\usepackage{bbold}
\usepackage{bm}
\usepackage{enumerate}
\usepackage{url}         
\usepackage{colonequals} 
\usepackage{authblk} 
\usepackage{a4wide}
\usepackage{graphicx}
\hypersetup{
    bookmarks=true,         
    unicode=false,          
    pdftoolbar=true,        
    pdfmenubar=true,        
    pdffitwindow=false,     
    pdfstartview={FitH},    
    pdftitle={My title},    
    pdfauthor={Author},     
    pdfsubject={Subject},   
    pdfcreator={Creator},   
    pdfproducer={Producer}, 
    pdfkeywords={keywords}, 
    pdfnewwindow=true,      
    colorlinks=true,       
    linkcolor=blue,          
    citecolor=blue,        
    filecolor=magenta,      
    urlcolor=cyan           
}

\newtheorem{theorem}{Theorem}[section]
\newtheorem{proposition}[theorem]{Proposition}
\newtheorem{corollary}[theorem]{Corollary}
\newtheorem{lemma}[theorem]{Lemma}
\theoremstyle{definition}
\newtheorem{remark}[theorem]{Remark}

\newtheorem{example}{Example}[section]

\def\!{\mathop{\mathrm{!}}}

\def\R{\mathbf{ R}}

\def\C{\mathcal{ C}}

\def\Pp{\mathbf{P}}

\newlength{\boxwidth}
\title{On the distribution of the number of internal equilibria in random evolutionary games}
\author[1]{Manh Hong Duong}
\author[2,3] {Hoang Minh Tran}
\author[4]{The Anh Han}
\affil[1]{Department of Mathematics, Imperial College London, London SW7 2AZ, 
UK.}
\affil[2]{Department of Industrial and Systems Engineering, Texas A\&M University, College Station, TX 77843. Email: tran@tamu.edu}
\affil[3]{School of Applied Mathematics and Informatics, Hanoi University of Science \& Technology, Hanoi, Vietnam}
\affil[4]{School of Computing, Media and Art, Teeside University, TS1 3BX, UK. Email: T.Han@tees.ac.uk}
\date{\today}
\begin{document}

\maketitle
\begin{abstract}
In this paper, we study the distribution of the number of internal equilibria of a multi-player two-strategy random evolutionary game. Using techniques from the random polynomial theory, we obtain a closed formula for the probability that the game has a certain number of internal equilibria. In addition, by  employing Descartes' rule of signs and combinatorial methods, we provide useful estimates for this probability. Finally, we also compare our analytical results with those obtained from samplings. 
\end{abstract}
\section{Introduction}
\subsection{Motivation}
Evolutionary Game Theory (EGT)~\cite{SP73} has become one of the most diverse and far reaching theories in biology finding applications in a plethora of disciplines such as ecology, population genetics, social sciences, economics  and computer science \cite{maynard-smith:1982to,axelrod:1984yo,hofbauer:1998mm,nowak:2006bo,broom2013game,Perc2010109,sandholm2010population,HanJaamas2016}. For example, in economics, EGT has been employed to make predictions in situations where traditional assumptions about agents' rationality and knowledge may not be justified \cite{friedman1998economic,sandholm2010population}. In computer science, EGT has been used extensively to model dynamics and emergent behaviour in multiagent systems \cite{tuyls2007evolutionary,HanBook2013}. Furthermore, EGT has provided explanations for the basic of cooperative behaviours which is one of the most well-studied and challenging interdisciplinary problems in science \cite{Pennisi93,hofbauer:1998mm,nowak:2006bo}. Of particular importance subclass in EGT is random evolutionary games  
in which the payoff entries are random variables. They are useful to model social and biological systems in which  very limited information is available, or where the environment changes so rapidly and frequently that one cannot predict the payoffs of their inhabitants \cite{may2001stability,fudenberg1992evolutionary,HTG12,gross2009generalized}.  In addition, as argued  in \cite{Galla2013}, even when random games are not  directly representative for real world scenarios, they are  valuable as a null hypothesis that can be used to sharpen our understanding of what makes real games special.

Similar to the foundational  concept of Nash equilibrium in classical game theory \cite{nash:1950ef}, the analysis of equilibrium points is of great importance in evolutionary game theory providing essential implications for understanding of complexity in a dynamical system, such as its behavioural, cultural or biological  diversity~\cite{broom:1997aa, broom:2003aa,gokhale:2010pn,HTG12, gokhale2014evolutionary, DH15, DuongHanJMB2016,Broom2016}. There are a considerable number of papers in the literature that study the number of equilibria, their stability and attainability in concrete strategic scenarios such as public goods games, see for example  \cite{broom:1997aa,Broom2000,Pacheco2009, Souza2009,Sasaki2015}. However, despite its importance, equilibrium properties in random games is far less understood with, to the best of our knowledge, only a few recent  efforts  \cite{gokhale:2010pn,HTG12, Galla2013, gokhale2014evolutionary, DH15, DuongHanJMB2016,Broom2016}. One of the most challenging problems in the study of equilibrium properties in random games is to characterise the distribution of the number of equilibria~\cite{gokhale:2010pn,HTG12}
\begin{center}
\textit{Can one compute the distribution of the number of (internal) equilibria in a $d$-player two-strategy random evolutionary game?}
\end{center}
In this paper, we address this question by providing closed formulas for the probability $p_m$ ($0\leq m\leq d-1$)that a $d$-player two strategy game has $m$ internal equilibria.

Using the replicator dynamics approach, to find an internal equilibrium in a $d$-player two-strategy game, one needs to solve the following polynomial equation for $y>0$ (see Equation  \eqref{eq: eqn for y} and its derivation in Section \ref{sec: replicator}),
\begin{equation}
\label{eq: P1}
P(y):=\sum\limits_{k=0}^{d-1}\beta_k\begin{pmatrix}
d-1\\
k
\end{pmatrix}y^k=0,
\end{equation}
where $\beta_k=a_k-b_k$, with $a_k$ and $b_k$ being  random variables representing  the entries of the game payoff matrix. In \cite{gokhale:2010pn,HTG12,gokhale2014evolutionary}, the authors provide both numerical and analytical results for games with  a small number of players ($d\leq 4$), focusing on the probability of attaining a maximal number of equilibrium points. These works use a direct approach by solving Equation \eqref{eq: P1},  expressing the positivity of its zeros as domains of conditions for the coefficients and then integrating over these domains to obtain the corresponding probabilities. However, in general, a polynomial of degree five or higher is not analytically solvable \cite{able:1824aa}. Therefore, it is impossible to extend the direct approach to the case of large number of players. In recent works~\cite{DH15, DuongHanJMB2016, DuongTranHan2017a} we have developed links between random evolutionary games and random polynomial theory~\cite{EK95} as well as the classical polynomial theory (particularly Legendre polynomials) employing techniques from the latter to study \textit{the expected number}, $E=\sum_{m=0}^{d-1} m p_m$, of internal equilibria. More specifically, we provided closed formulas for $E$, characterized its asymptotic limits as $d$ tends to infinity and investigated the effect of correlation in the case of correlated payoff entries.  The derivation of the individual probabilities $p_m$ ($0\leq m\leq d-1$) is harder than that of the expectation; however, it will provide more insights into the understanding of the equilibrium properties of the game such as the probabilities of having no/unique/maximal number of internal equilibria. In this paper we explore deeper links between random polynomial theory and random evolutionary game theory. 
\subsection{Summary of main results}
We now summarize our main results. Detailed statements and proofs will be given in the sequel sections. The first main result of the paper is explicit formulas for the distribution of the number of internal equilibria. 
\begin{theorem}[The distribution of the number of internal equilibria in a $d$-player two-strategy random evolutionary game]
\label{theo: main theo 1}
Suppose that the coefficients $\{\beta_k\}$ in~\eqref{eq: P1} are either normally distributed or uniformly distributed or being the difference of uniformly distributed random variables.  The probability that a $d$-player two-strategy random evolutionary game has $m$, $0\leq m\leq d-1$,  internal equilibria, is given by 
\begin{equation}
\label{eq: pm1}
p_{m}=\sum_{k=0}^{\lfloor \frac{d-1-m}{2}\rfloor}p_{m,2k,d-1-m-2k},
\end{equation}
where $p_{m,2k,d-1-m-2k}$ are given in \eqref{eq: pm2kG}-\eqref{eq: pm2kU1}-\eqref{eq: pm2kU2} respectively.
\end{theorem}
In Section~\ref{sec: distribution}, we will derive Theorem~\ref{theo: main theo 1} from a more general theorem, Theorem~\ref{theo: Zap06}, where we provide explicit formulas for the probability $p_{m,2k,n-m-2k}$ that a random polynomial or degree $n$ has $m$ ($0\leq m\leq n$) positive, $2k$ ($0\leq k\leq \lfloor\frac{n-m}{2}\rfloor$) complex and $n-2m-2k$ negative roots. We expect that Theorem~\ref{theo: Zap06} is also useful in a more general context in the random polynomial theory. 

Theorem~\ref{theo: main theo 1} is theoretically interesting; however, to obtain $p_m$ one needs to calculate all the probabilities $p_{m,2k,d-1-2k}$, $ 0\leq k\leq \lfloor\frac{n-m}{2}\rfloor$, which are complex multiple integrals. In Section~\ref{sec: numerics} we do compute all the probabilities $p_m$, $0\leq m\leq d-1$, for $d$ is up to $5$ and compare the result with the sampling method. When $d$ is larger it becomes computationally expensive to compute these probabilities using formula~\eqref{eq: pm1}. Our second main result offers simpler explicit estimates of $p_m$ in terms of $d$ and $m$. It turns out that the symmetry of the coefficients $\beta_k$ plays a significant role. We consider two cases
\[
\text{Case 1}:~~\Pp(\beta_k>0)=\Pp(\beta_k<0)=\frac{1}{2}\quad\text{and Case 2}:~~\Pp(\beta_k>0)=\Pp(\beta_k<0)=\alpha
\]
for all $k=0,\ldots, d-1$ and for some $0\leq \alpha\leq 1$. Case 1 is obviously a particular instance of Case 2 when $\alpha=\frac{1}{2}$; however, due to its special symmetric property, we will provide a much more simpler treatment for Case 1 than the general one.
\begin{theorem}
\label{theo: main theo 2}
We have the following estimates for $p_m$
\begin{equation}
\label{eq: est of pm}
p_m\leq \sum_{\substack{k\geq m\\ k-m~\text{even}}} p_{k,d-1},
\end{equation}
where $p_{k,d-1}=\frac{1}{2^{d-1}}\begin{pmatrix}
d-1\\k
\end{pmatrix}$ if $\alpha=\frac{1}{2}$, in this case the sum on the right hand side of ~\eqref{eq: est of pm} can be computed explicitly in terms of $m$ and $d$, and can be computed explicitly according to Theorem \ref{theo: explicit p} for the general case.  Estimate~\eqref{eq: est of pm} has several interesting consequences such as explicit bounds for $p_{d-2}$ and $p_{d-1}$ as well as the following assertions
\begin{enumerate}[1)]
\item For $d=2$: $p_0=\alpha^2+(1-\alpha)^2$ and $p_1=2\alpha(1-\alpha)$;
\item For $d=3$: $p_1=2\alpha(1-\alpha)$.
\end{enumerate}
\end{theorem}
We will derive Theorem~\ref{theo: main theo 2} in Section~\ref{sec: universal} from Descartes' rule of signs applied to random polynomials and combinatorial techniques. Our technique can be used to obtain estimates for the probability that a random polynomial has a certain number of positive roots which is an interesting problem on its own right. 
\subsection{Organisation of the paper}
The rest of the paper is organised as follows.  In Section \ref{sec: replicator}, we summarize the replicator dynamics for multi-player two-strategy games. Section~\ref{sec: distribution} is devoted to the proof of Theorem~\ref{theo: main theo 1} on the probability distribution. The proof of Theorem~\ref{theo: main theo 2} will be given in Section~\ref{sec: universal}. In Section~\ref{sec: numerics} we show some numerical simulations to demonstrate analytical results. Finally, further discussions are given in Section~\ref{sec: discussion}.
\section{Replicator dynamics}
\label{sec: replicator}
A fundamental model of evolutionary game theory  is replicator dynamics  \cite{taylor:1978wv,zeeman:1980ze,hofbauer:1998mm,schuster:1983le,nowak:2006bo}, describing that whenever a strategy has a fitness larger than the  average fitness of the population, it is expected to  spread. For the sake of completeness, below we derive the replicator dynamics for multi-player two-strategy games.

Consider an infinitely large population with two strategies, A and B. Let  $x$, $0 \leq x \leq 1$, be the frequency of strategy A. The frequency of strategy B is thus $(1-x)$. The interaction of the individuals in the population is in randomly selected groups of $d$ participants, that is, they play and obtain their fitness from   $d$-player games. The game is defined through a $(d-1)$-dimensional payoff matrix \cite{gokhale:2010pn}, as follows. Let $a_k$  (resp., $b_k$) be the payoff of an A-strategist (resp., B) in a group  containing  $k$ A strategists (i.e. $d-k$ B strategists). In this paper, we consider symmetric games where the payoffs do not depend on the ordering of the players. Asymmetric games will be studied in our forthcoming paper~\cite{DuongTranHan2017c}. In the symmetric case, the probability that an A strategist interacts with $k$ other A strategists in a group of size $d$ is 
\begin{equation}
\begin{pmatrix}
d-1\\
k
\end{pmatrix}x^k (1-x)^{d-1-k}.
\end{equation}
Thus, the average payoffs of $A$ and $B$  are, respectively 
\begin{equation*}
\pi_A= \sum\limits_{k=0}^{d-1}a_k\begin{pmatrix}
d-1\\
k
\end{pmatrix}x^k (1-x)^{d-1-k},	\quad
\pi_B = \sum\limits_{k=0}^{d-1}b_k\begin{pmatrix}
d-1\\
k
\end{pmatrix}x^k (1-x)^{d-1-k}.
\end{equation*}
The replicator equation of a $d$-player two-strategy game is given by  \cite{hofbauer:1998mm,sigmund:2010bo,gokhale:2010pn}
\begin{equation*}
\dot{x}=x(1-x)\big(\pi_A-\pi_B\big).
\end{equation*}
Since $x=0$ and $x=1$ are two trivial equilibrium points, we focus only on internal ones,  i.e. $0 < x < 1$. They satisfy the condition that the  fitnesses of both strategies are the same $\pi_A=\pi_B$, which gives rise to
\begin{equation*}
\sum\limits_{k=0}^{d-1}\beta_k \begin{pmatrix}
d-1\\
k
\end{pmatrix}x^k (1-x)^{d-1-k} = 0,
\end{equation*}
where $\beta_k = a_k - b_k$. Using the transformation $y= \frac{x}{1-x}$, with $0< y < +\infty$, dividing the left hand side of the above equation by $(1-x)^{d-1}$ we obtain the following polynomial equation for $y$
\begin{equation}
\label{eq: eqn for y}
P(y):=\sum\limits_{k=0}^{d-1}\beta_k\begin{pmatrix}
d-1\\
k
\end{pmatrix}y^k=0.
\end{equation}
Note that this equation can also be derived from the definition of an evolutionary stable strategy, see e.g., \cite{broom:1997aa}. As in \cite{gokhale:2010pn, DH15, DuongHanJMB2016}, we are interested in random games where $a_k$ and $b_k$ (thus $\beta_k$), for $0\leq k\leq d-1 $, are random variables. 

In Section~\ref{sec: distribution} where we provide estimates for the number of internal equilibria in a $d$-player two-strategy game, we will use the information on the symmetry of $\beta_k$. The following lemma gives a necessary condition to determine when the difference of two random variables is symmetrically distributed.
  
\begin{lemma}\cite[Lemma 3.5]{DuongTranHan2017a}
\label{lem: symmetry of betas} Let $X$ and $Y$ be two exchangeable random variables, i.e. their joint probability distribution $f_{X,Y}(x,y)$ is symmetric, $f_{X,Y}(x,y)=f_{X,Y}(y,x)$. Then $Z=X-Y$ is symmetrically distributed about $0$, i.e., its probability distribution satisfies $f_Z(z)=f_Z(-z)$. In addition, if $X$ and $Y$ are iid then they are exchangeable.
\end{lemma}

\section{The distribution of the number of positive zeros of random polynomial and applications to EGT}
\label{sec: distribution}

In this section, we are interested in finding the distribution of the number of internal equilibria of a $d$-player two-strategy random evolutionary game. We recall that an internal equilibria is a real and positive zero of the polynomial $P(y)$ in \eqref{eq: eqn for y}
\begin{equation}
\label{eq: P}
P(y):=\sum\limits_{k=0}^{d-1}\beta_k\begin{pmatrix}
d-1\\
k
\end{pmatrix}y^k=0.
\end{equation}
Denote by $\kappa$ the number of positive zeros of this polynomial. For a given $m$, $0\leq m\leq d-1$, we need to compute the probability $p_m$ that $\kappa=m$. To this end, we first adapt a recent method introduced in \cite{Zap06} (see also \cite{ButezZeitouni2017, GKZ2017TM} for its applications to other problems) to establish a formula to compute the probability that a general random polynomial has a given number of real and positive zeros. Then we apply the general theory to the polynomial \eqref{eq: P} above.
\subsection{The distribution of the number of positive zeros of a random polynomials}

Consider a general random polynomial
\begin{equation}
\label{eq: general polynomial}
\Pp(t)=\xi_0 t^n+\xi_1t^{n-1}+\ldots+\xi_{n-1}t+\xi_n.
\end{equation}
We use the following notations for the elementary symmetric polynomials
\begin{align}
\sigma_0(y_1,\ldots,y_n)&=1,\nonumber
\\\sigma_1(y_1,\ldots,y_n)&=y_1+\ldots+y_n,\nonumber
\\\sigma_2(y_1,\ldots,y_n)&=y_1y_2+\ldots+y_{n-1}y_n\label{eq: sym pols}
\\&\vdots\nonumber
\\\sigma_{n-1}(y_1,\ldots,y_n)&=y_1y_2\ldots y_{n-1}+\ldots+y_2y_3\ldots y_n,\nonumber
\\ \sigma_{n}(y_1,\ldots,y_n)&=y_1\ldots y_n;\nonumber
\end{align}
and denote
\begin{equation}
\label{eq: Delta}
\Delta(y_1,\ldots,y_n)=\prod_{1\leq i<j\leq n}|y_i-y_j|.
\end{equation}
the Vandermonde determinant. 
\begin{theorem}
\label{theo: Zap06}
Assume that the random variables $\xi_0,\xi_1,\ldots, \xi_n$ have a joint density $p(a_0,\ldots,a_n)$. Let $0\leq m\leq d-1$ and $0\leq k\leq \lfloor \frac{n-m}{2}\rfloor$. The probability $p_{m,2k,n-m-2k}$ that $\Pp$ has $m$ positive, $2k$ complex and $n-m-2k$ negative zeros is given by
\begin{multline}
\label{eq: pm2k}
p_{m,2k,n-m-2k}=\frac{2^{k}}{m! k! (n-m-2k)!}\int_{\R_+^m}\int_{\R_-^{n-m-2k}} 
\int_{\R_+^k}\int_{[0,\pi]^k}\int_{\R}
\\ r_1\ldots r_k p(a\sigma_0,\ldots,a\sigma_{n}) |a^{n}\Delta|\, da\,d\alpha_1\ldots d\alpha_k dr_1\ldots dr_k dx_1\ldots dx_{n-2k},
\end{multline}
where 
\begin{align}
\label{eq:sigma}
&\sigma_j=\sigma_j(x_1,\ldots,x_{n-2k}, r_1e^{i\alpha_1}, r_1e^{-i\alpha_1},\ldots,r_k e^{i\alpha_k}, r_k e^{-i \alpha_k}),
\\&\Delta=\Delta(x_1,\ldots,x_{n-2k}, r_1e^{i\alpha_1}, r_1e^{-i\alpha_1},\ldots,r_k e^{i\alpha_k}, r_k e^{-i \alpha_k}).\label{eq:Delta}
\end{align}
As consequences,
\begin{enumerate}[(1)]
\item The probability that $\Pp$ has $m$ positive zeros is
\begin{equation}
p_{m}=\sum_{k=0}^{\lfloor \frac{n-m}{2}\rfloor}p_{m,2k,n-m-2k}.
\end{equation}
\item In particular, the probability that $\Pp$ has the maximal number of positive zeros is
\begin{equation}
p_{n}=\frac{2^{k}}{k! (n-2k)!}\int_{\R_+^{n}}\int_{\R}p(a\sigma_0,\ldots,a\sigma_{n})\, |a^{n}\,\Delta|\, dadx_1\ldots dx_{n},
\end{equation}
where
\begin{equation*}
\sigma_j=\sigma_j(x_1,\ldots,x_{n}),\quad\Delta=\Delta(x_1,\ldots,x_{n}).
\end{equation*}
\end{enumerate}
\end{theorem}
\begin{proof}
The reference \cite[Theorem 1]{Zap06} provides a formula to compute the probability that the polynomial $\Pp$ has $n-2k$ real and $2k$ complex roots. In the present paper, we need to distinguish between positive and negative real zeros. We now sketch and adapt the proof of \cite[Theorem 1]{Zap06} to obtain the formula \eqref{eq: pm2k} for the probability that the polynomial $\Pp$ has $m$ positive, $2k$ complex and $n-m-2k$ negative roots. Consider a $(n+1)$-dimensional vector space $\mathbf{V}$ of polynomials of the form
\[
Q(t)=a_0t^n+a_1t^{n-1}+\ldots+a_{n-1}t+a_n
\]
and a measure $\mu$ on this space defined as the integral of the differential form
\begin{equation}
\label{eq: mu}
dQ=p(a_0,\ldots,a_n)\,da_0\wedge\ldots\wedge da_n.
\end{equation}
Our goal is to find $\mu(V_{m,2k})$ where $V_{m,2k}$ is the set of polynomials having $m$ positive, $2k$ complex and $n-m-2k$ negative roots. Let $Q\in V_{m,2k}$. Denote all zeros of $Q$ as
\begin{align*}
z_1=x_1,\ldots,z_{n-2k}=x_{n-2k},~z_{n-2k+1}=r_1 e^{i \alpha_1},~z_{n-2k+2}=r_1 e^{-i \alpha_1},\ldots,z_{n-1}=r_k e^{i \alpha_k},~z_{n}=r_k e^{-i \alpha_k},
\end{align*}
where
\[
0<x_1,\ldots, x_m<\infty;\quad -\infty <x_{m+1},\ldots, x_{n-2k}<0; \quad 0<r_1,\ldots,r_k<\infty;\quad 0<\alpha_1,\ldots,\alpha_k<\pi.
\]
To find $\mu(V_{m,2k})$ we need to integrate the differential form \eqref{eq: mu} over the set $V_{m,2k}$. The key idea in the proof of \cite[Theorem 1]{Zap06} is to make a change of coordinates $(a_0,\ldots, a_n)\mapsto (a_0,x_1,\ldots,x_{n-2k}, r_1,\ldots, r_k, \alpha_1,\ldots, \alpha_k)$ and find $dQ$ in the new coordinates. The derivation of the following formula is carried out in detail in \cite[Theorem 1]{Zap06}:
\begin{multline}
dQ=2^k r_1\ldots r_k\, p(a_0,a_0\sigma_1(x_1,\ldots,x_{n-2k},r_1e^{i\alpha_1},r_1e^{-i\alpha_1},\ldots,r_ke^{i\alpha_k}, r_ke^{-i\alpha_k}),\ldots\\ a_0\sigma_n(x_1,\ldots,x_{n-2k},r_1e^{i\alpha_1},r_1e^{-i\alpha_1},\ldots,r_ke^{i\alpha_k}, r_ke^{-i\alpha_k}))
\\ \times |a_0^n \Delta ((x_1,\ldots,x_{n-2k},r_1e^{i\alpha_1},r_1e^{-i\alpha_1},\ldots,r_ke^{i\alpha_k}, r_ke^{-i\alpha_k}))|
\\ \times dx_1\wedge\ldots\wedge dx_{n-2k}\wedge dr_1\wedge\ldots\wedge dr_k\wedge d\alpha_1\wedge\ldots\wedge d\alpha_k\wedge da_0.
\end{multline}
Now we integrate this equation over all polynomials $Q$ that have $m$ positive zeros, $n-m-2k$ negative zeros and $k$ complex zeros in the upper half-plane. Since there are $m !$ permutations of the positive zeros, $(n-m-2k)!$ permutations of the negative zeros, and $k!$ permutations of the complex zeros, after integrating each polynomial in the left-hand side will occur $m!k!(n-m-2k)!$ times. Hence the integral of the left-hand side is equal to $m!k!(n-m-2k)! \, p_{m,2k,n-m-2k}$. The integral on the right-hand side equals 
\begin{multline}
2^k\int_{\R_+^m}\int_{\R_-^{n-m-2k}} 
\int_{\R_+^k}\int_{[0,\pi]^k}\int_{\R} r_1\ldots r_k p(a\sigma_0,\ldots,a\sigma_{n}) |a^{n}\Delta|\, da\,d\alpha_1\ldots d\alpha_k dr_1\ldots dr_k dx_1\ldots dx_{n-2k},
\end{multline}
hence the assertion \eqref{eq: pm2k} follows.
\end{proof}
\subsection{The distribution of the number of internal equilibria}
Next we apply Theorem \ref{theo: Zap06} to compute the probability that a random evolutionary game has $m$, $0\leq m\leq d-1$, internal equilibria. We derive  formulas for the most three common cases \cite{HTG12}:
\begin{enumerate}[C1)]
\item $\{\beta_j,0\leq j\leq d-1\}$ are i.i.d. standard normally distributed,
\item $\{\beta_j\}$ are i.i.d uniformly distributed with the common distribution $f_j(x)=\frac{1}{2} \mathbb{1}_{[-1,1]}(x)$,
\item $\{a_k\}$ and $\{b_k\}$  are i.i.d uniformly distributed with the common distribution $f_j(x)=\frac{1}{2} \mathbb{1}_{[-1,1]}(x)$.
\end{enumerate}

The main result of this section is the following theorem (cf. Theorem \ref{theo: main theo 2}).
\def\bs{\boldsymbol{\sigma}}
\begin{theorem}
\label{prop: pm}
The probability that a $d$-player two-strategy random evolutionary game has $m$ ($0\leq m\leq d-1$) internal equilibria is
\begin{equation}
p_{m}=\sum_{k=0}^{\lfloor \frac{d-1-m}{2}\rfloor}p_{m,2k,d-1-m-2k},
\end{equation}
where $p_{m,2k,d-1-m-2k}$ is given below for each of the case above:

- for the case C1)
\begin{multline}
\label{eq: pm2kG}
p_{m,2k,d-1-m-2k}\\=\frac{2^{k}}{m! k! (d-1-m-2k)!}~\frac{ \Gamma\Big(\frac{d}{2}\Big) }{(\pi)^{\frac{d}{2}}\prod\limits_{i=0}^{d-1}\delta_i}\int_{\R_+^m}\int_{\R_-^{d-1-2k-m}} 
\int_{\R_+^k}\int_{[0,\pi]^k}\, r_1\ldots r_k\, \left(\sum\limits_{i=0}^{d-1}\frac{\sigma_i^2}{\delta_i^2}\right)^{-\frac{d}{2}}\Delta\\d\alpha_1\ldots d\alpha_k dr_1\ldots dr_k dx_1\ldots dx_{d-1-2k}
\end{multline}
where $\sigma_i$, for $i=0,\ldots,d-1$, and $\Delta$ are given in \eqref{eq:sigma}--\eqref{eq:Delta} and
$\delta_i=\begin{pmatrix}
d-1\\i
\end{pmatrix}$.

-for the case C2) 
\begin{multline}
\label{eq: pm2kU1}
p_{m,2k,d-1-m-2k}=\frac{2^{k+1-d}}{d \, m!\, k! \,(d-1-m-2k)!\prod\limits_{i=0}^{d-1} \delta_i}\int_{\R_+^m}\int_{\R_-^{d-1-2k-m}} 
\int_{\R_+^k}\int_{[0,\pi]^k}\, r_1\ldots r_k\,\Big(\min\big\{|\delta_i/\sigma_i|\big\}\Big)^{d} \Delta\\ \,d\alpha_1\ldots d\alpha_k dr_1\ldots dr_k dx_1\ldots dx_{d-1-2k}.
\end{multline}

--for the case C3)
\begin{multline}
\label{eq: pm2kU2}
p_{m,2k,d-1-m-2k}=\frac{2^{k+1}(-1)^d}{m! k! (d-1-m-2k)!\prod_{j=0}^{d-1}\delta_j^2}\int_{\R_+^m}\int_{\R_-^{d-1-2k-m}} 
\int_{\R_+^k}\int_{[0,\pi]^k}
\\  r_1\ldots r_k\, \prod_{j=0}^{d-1}|\sigma_j|\sum_{i=0}^{d}(-1)^i \frac{K_i}{2d-i} \Big(\min\big\{|\delta_i/\sigma_i|\big\}\Big)^{2d-i}\Delta\,d\alpha_1\ldots d\alpha_k dr_1\ldots dr_k dx_1\ldots dx_{d-1-2k}.
\end{multline}

In particular, the probability that a $d$-player two-strategy random evolutionary game has the maximal number of internal equilibria is

1) for the case C1)
\begin{equation}
\label{eq: pmax1}
p_{d-1}=\frac{1}{(d-1)!}~\frac{\Gamma\Big(\frac{d}{2}\Big) }{(\pi)^\frac{d}{2} \prod\limits_{i=0}^{d-1}\delta_i}~\int_{\R_+^{d-1}} q(\sigma_0,\ldots,\sigma_{d-1})\,dx_1\ldots dx_{d-1},
\end{equation}

2) for the case C2)
\begin{equation}
\label{eq: pmax2}
p_{d-1}=\frac{2^{1-d}}{d! \prod_{i=0}^{d-1} \delta_i}~\int_{\R_+^{d-1}}\Big(\min\big\{|\delta_i/\sigma_i|\big\}\Big)^{d} \Delta \,dx_1\ldots dx_{d-1},
\end{equation}

3) for the case C3)
\begin{equation}
\label{eq: pmax3}
p_{d-1}=\frac{2(-1)^d}{(d-1)!\prod_{j=0}^{d-1}\delta_j^2}\int_{\R_+^{d-1}}\prod_{j=0}^{d-1}|\sigma_j|\sum_{i=0}^{d}(-1)^i \frac{K_i}{2d-i} \Big(\min\big\{|\delta_i/\sigma_i|\big\}\Big)^{2d-i}\Delta\,dx_1\ldots dx_{d-1}.
\end{equation}
Note that in formulas \eqref{eq: pmax1}-\eqref{eq: pmax3} above, $\sigma_j=\sigma_j(x_1,\ldots,x_{d-1}),\quad \Delta=\Delta(x_1,\ldots,x_{d-1})$.
\end{theorem}
\begin{proof}
\def\y{\mathbf{y}}
1) Since $\{\beta_j,0\leq j\leq d-1\}$ are i.i.d. standard normally distributed and for $i\neq j$ the joint distribution $p(y_0,\ldots,y_{d-1})$ of $\left\{\begin{pmatrix}
d-1\\j
\end{pmatrix}\beta_j,0\leq j\leq d-1\right\}$ is given by
\[
p(y_0,\ldots,y_{d-1})=\frac{1}{(2\pi)^{\frac{d}{2}} \prod_{i=0}^{d-1}\begin{pmatrix}
d-1\\i
\end{pmatrix}}\exp\left[-\frac{1}{2}\sum_{i=0}^{d-1}\frac{y_i^2}{\begin{pmatrix}
d-1\\i
\end{pmatrix}^2}\right]=\frac{1}{(2\pi)^{\frac{d}{2}}|\C|^\frac{1}{2}}\exp\Big[-\frac{1}{2}\y^T\C^{-1}\y\Big],
\]
where $\mathbf{y}=[y_0~~y_1~~\ldots~~y_{d-1}]^T$ and $\C$ is the covariance matrix 
\[
\C_{ij}=\begin{pmatrix}
d-1\\i
\end{pmatrix}\begin{pmatrix}
d-1\\j
\end{pmatrix}\delta_{ij}.
\]
Therefore,
\begin{equation}
\label{eq: formula of p}
p(a\sigma_0,\ldots, a \sigma_{d-1})=\frac{1}{(2\pi)^\frac{d}{2} |\C|^\frac{1}{2}}\exp\Bigg(-\frac{a^2}{2}\bs^T\,\C^{-1}\,\bs\Bigg)\quad\text{where}\quad
\bs=[\sigma_0~\sigma_1~\ldots~\sigma_{d-1}]^T.
\end{equation}
Using the following formula for moments of a normal distribution, 
\begin{equation*}
\int_{\R}|x|^n\exp\big(-\alpha x^2\big)\,dx=\frac{\Gamma\big(\frac{n+1}{2}\big)}{\alpha^\frac{n+1}{2}},
\end{equation*}
we compute
\begin{equation*}
\int_{\R}|a|^{d-1}\exp\Bigg(-\frac{a^2}{2}\bs^T\,\C^{-1}\,\bs\Bigg)\,da=\frac{\Gamma\Big(\frac{d}{2}\Big)}{\Big(\frac{\bs^T\C^{-1}\bs}{2}\Big)^{\frac{d}{2}}}=\frac{2^\frac{d}{2}\Gamma\Big(\frac{d}{2}\Big)}{\big(\bs^T\C^{-1}\bs\big)^{\frac{d}{2}}}.
\end{equation*}
Applying Theorem \ref{theo: Zap06} to the polynomial $P$ given in \eqref{eq: P1} and using the above identity we obtain
\begin{align*}
p_{m,2k,d-1-m-2k}&=\frac{2^{k}}{m! k! (d-1-m-2k)!}\int_{\R_+^m}\int_{\R_-^{d-1-2k-m}} 
\int_{\R_+^k}\int_{[0,\pi]^k}\int_{\R}
\\& \qquad \, r_1\ldots r_k\, p(a\sigma_0,\ldots,a\sigma_{d-1}) |a|^{d-1}\Delta\, da\,d\alpha_1\ldots d\alpha_k dr_1\ldots dr_k dx_1\ldots dx_{d-1-2k}
\\&=\frac{2^{k}}{m! k! (d-1-m-2k)!}~\frac{1}{(2\pi)^\frac{d}{2} |\C|^\frac{1}{2}}~ 2^{\frac{d}{2}}\Gamma\Big(\frac{d}{2}\Big) ~\int_{\R_+^m}\int_{\R_-^{d-1-2k-m}} 
\int_{\R_+^k}\int_{[0,\pi]^k}
\\& \qquad \, r_1\ldots r_k\, \big(\bs^T\C^{-1}\bs\big)^{-\frac{d}{2}}~\Delta\,d\alpha_1\ldots d\alpha_k dr_1\ldots dr_k dx_1\ldots dx_{d-1-2k}
\\&=\frac{2^{k}}{m! k! (d-1-m-2k)!}~\frac{\Gamma\Big(\frac{d}{2}\Big) }{(\pi)^\frac{d}{2} |\C|^\frac{1}{2}}~\int_{\R_+^m}\int_{\R_-^{d-1-2k-m}} 
\int_{\R_+^k}\int_{[0,\pi]^k}
\\& \qquad \, r_1\ldots r_k\, \big(\bs^T\C^{-1}\bs\big)^{-\frac{d}{2}}~\Delta\,d\alpha_1\ldots d\alpha_k dr_1\ldots dr_k dx_1\ldots dx_{d-1-2k},
\end{align*}
which is the desired equality \eqref{eq: pm2kG} by definition of $\C$ and $q$.

2) Now since $\{\beta_j\}$ are i.i.d uniformly distributed with the common distribution $f_j(x)=\frac{1}{2} \mathbb{1}_{[-1,1]}(x)$, the joint distribution $p(y_0,\ldots,y_{d-1})$ of $\left\{\begin{pmatrix}
d-1\\j
\end{pmatrix}\beta_j,0\leq j\leq d-1\right\}$ is given by
\[
p(y_0,\ldots,y_{d-1})=\frac{1}{2^{d}\prod_{i=0}^{d-1} \delta_i}\mathbb{1}_{\bigtimes_{i=0}^{d-1}[-\delta,\delta_i]}(y_0,\ldots, y_{d-1}) \quad\text{where}~~\delta_i=\begin{pmatrix}
d-1\\i
\end{pmatrix}.
\]
Therefore,
\[
p(a\sigma_0,\ldots,a \sigma_{d-1})=\frac{1}{2^{d}\prod_{i=0}^{d-1} \delta_i}\mathbb{1}_{\bigtimes_{i=0}^{d-1}[-\delta,\delta_i]}(a\sigma_0,\ldots, a \sigma_{d-1}).
\]
Since $\mathbb{1}_{\bigtimes_{i=0}^{d-1}[-\delta,\delta_i]}(a\sigma_0,\ldots, a \sigma_{d-1})=1$ if and only if $a\sigma_i\in [-\delta_i,\delta_i]$ for all $i=0,\ldots, d-1$; i.e.,
\[
a\in \bigcap\limits_{i=0}^{d-1} \big[-|\delta_i/\sigma_i|,|\delta_i/\sigma_i|\big]=\big[-\min\big\{|\delta_i/\sigma_i|\big\},\min\big\{|\delta_i/\sigma_i|\big\}\big],
\]
We compute
\begin{align*}
\int_{\R}|a|^{d-1}p(a\sigma_0,\ldots,a \sigma_{d-1})\,da=\frac{1}{2^{d}\prod_{i=0}^{d-1} \delta_i} \int_{-\min\big\{|\delta_i/\sigma_i|\big\}}^{\min\big\{|\delta_i/\sigma_i|\big\}}|a|^{d-1}\,da=\frac{1}{d\, 2^{d-1}\prod_{i=0}^{d-1} \delta_i} \Big(\min\big\{|\delta_i/\sigma_i|\big\}\Big)^{d}.
\end{align*}
Similarly as in the Gaussian case, using this identity and applying Theorem \ref{theo: Zap06} we obtain
\begin{align*}
p_{m,2k,d-1-m-2k}&=\frac{2^{k}}{m! k! (d-1-m-2k)!}\int_{\R_+^m}\int_{\R_-^{d-1-2k-m}} 
\int_{\R_+^k}\int_{[0,\pi]^k}\int_{\R}
\\& \qquad \, r_1\ldots r_k\, p(a\sigma_0,\ldots,a\sigma_{d-1}) |a|^{d-1}\Delta\, da\,d\alpha_1\ldots d\alpha_k dr_1\ldots dr_k dx_1\ldots dx_{d-1-2k}
\\&=\frac{2^{k+1-d}}{d \, m!\, k! \,(d-1-m-2k)! \prod_{i=0}^{d-1} \delta_i}\int_{\R_+^m}\int_{\R_-^{d-1-2k-m}} 
\int_{\R_+^k}\int_{[0,\pi]^k}
\\& \qquad \, r_1\ldots r_k\, \Big(\min\big\{|\delta_i/\sigma_i|\big\}\Big)^{d} \Delta\, da\,d\alpha_1\ldots d\alpha_k dr_1\ldots dr_k dx_1\ldots dx_{d-1-2k}.
\end{align*}
3) Now we assume that $a_j$ and $b_j$ are i.i.d uniformly distributed with the common distribution $\gamma(x)=\frac{1}{2} \mathbb{1}_{[-1,1]}(x)$. Since $\beta_j=a_j-b_j$, its probability density is given by
\begin{align*}
\gamma_{\beta}(x)=\int_{-\infty}^{+\infty}f(y)f(x+y)\,dy=(1-|x|)\mathbb{1}_{[-1,1]}(x).
\end{align*}
The probability density of $\delta_j\beta_j$ is
\[
\gamma_{j}(x)=\frac{1}{\delta_j}\Big(1-\frac{|x|}{\delta_j}\Big)\mathbb{1}_{[-1,1]}(x/\delta_j)=\frac{\delta_j-|x|}{\delta_j^2}\mathbb{1}_{[-\delta_j,\delta_j]}(x),
\]
and the joint distribution $p(y_0,\ldots,y_{d-1})$ of $\left\{\delta_j\beta_j,0\leq j\leq d-1\right\}$ is given by
\[
p(y_0,\ldots,y_{d-1})=\prod\limits_{j=0}^{d-1}\frac{\delta_j-|y_j|}{\delta_j^2}\mathbb{1}_{\bigtimes_{i=0}^{d-1}[-\delta_i,\delta_i]}(y_0,\ldots, y_{d-1}).
\]
Therefore
\[
p(a\sigma_0,\ldots,a \sigma_{d-1})=\prod\limits_{j=0}^{d-1}\frac{\delta_j-|a \sigma_j|}{\delta_j^2}\mathbb{1}_{\bigtimes_{i=0}^{d-1}[-\delta_i,\delta_i]}( a \sigma_0,\ldots,a \sigma_{d-1}).
\]
We compute
\begin{align*}
\int_{\R}|a|^{d-1}p(a\sigma_0,\ldots,a \sigma_{d-1})\,da&=\frac{1}{\prod_{j=0}^{d-1}\delta_j^2}\int_{-\min\big\{|\delta_i/\sigma_i|\big\}}^{\min\big\{|\delta_i/\sigma_i|\big\}}|a|^{d-1}\prod_{j=0}^{d-1}(\delta_j-|a\sigma_j|)\,da
\\&=\frac{2}{\prod_{j=0}^{d-1}\delta_j^2}\int_{0}^{\min\big\{|\delta_i/\sigma_i|\big\}}a^{d-1}\prod_{j=0}^{d-1}(\delta_j-a|\sigma_j|)\,da
\\&=2 (-1)^d \prod_{j=0}^{d-1}\frac{|\sigma_j|}{\delta_j^2}\int_{0}^{\min\big\{|\delta_i/\sigma_i|\big\}}a^{d-1}\prod_{j=0}^{d-1}\Big(a-\frac{\delta_j}{|\sigma_j|}\Big)\,da
\\&=2 (-1)^d \prod_{j=0}^{d-1}\frac{|\sigma_j|}{\delta_j^2}\sum_{i=0}^{d}(-1)^i K_i\int_{0}^{\min\big\{|\delta_i/\sigma_i|\big\}}a^{2d-1-i}\,da
\\&=2 (-1)^d \prod_{j=0}^{d-1}\frac{|\sigma_j|}{\delta_j^2}\sum_{i=0}^{d}(-1)^i \frac{K_i}{2d-i} \Big(\min\big\{|\delta_i/\sigma_i|\big\}\Big)^{2d-i}
\end{align*}
where $K_i=\sigma_i(\delta_0/|\sigma_0|,\ldots, \delta_{d-1}/|\sigma_{d-1}|)$ for $i=0,\ldots, d$.

Therefore,
\begin{align*}
&p_{m,2k,d-1-m-2k}\nonumber
\\&=\frac{2^{k}}{m! k! (d-1-m-2k)!}\int_{\R_+^m}\int_{\R_-^{d-1-2k-m}} 
\int_{\R_+^k}\int_{[0,\pi]^k}\int_{\R}\nonumber
\\& \qquad \, r_1\ldots r_k\, p(a\sigma_0,\ldots,a\sigma_{d-1}) |a|^{d-1}\Delta\, da\,d\alpha_1\ldots d\alpha_k dr_1\ldots dr_k dx_1\ldots dx_{d-1-2k}\nonumber
\\&=\frac{2^{k+1}(-1)^d}{m! k! (d-1-m-2k)!\prod_{j=0}^{d-1}\delta_j^2}\int_{\R_+^m}\int_{\R_-^{d-1-2k-m}} 
\int_{\R_+^k}\int_{[0,\pi]^k}\nonumber
\\& \qquad \, r_1\ldots r_k\, \prod_{j=0}^{d-1}|\sigma_j|\sum_{i=0}^{d}(-1)^i \frac{K_i}{2d-i} \Big(\min\big\{|\delta_i/\sigma_i|\big\}\Big)^{2d-i}\Delta\,d\alpha_1\ldots d\alpha_k dr_1\ldots dr_k dx_1\ldots dx_{d-1-2k}
\end{align*}
\end{proof}

\begin{corollary} The expected numbers  of the internal equilibria and stable internal equilibria, $E(d)$ and $SE(d)$, respectively,  of a $d$-player two-strategy can be computed via
\begin{equation*}
E(d)=\sum_{m=0}^{d-1} m p_m, \quad \quad  SE(d)=\frac{1}{2}\sum_{m=0}^{d-1} m p_m.
\end{equation*}
Note that these formulas are applicable for non-Gaussian distributions in contrast to the method used in the previous section that can only be  used  for the Gaussian case.
\end{corollary}

\begin{remark}
\label{remark:indepedenceOfVariance}
In Theorem \ref{prop: pm} for the case C1), the assumption that $\beta_k$ are standard Gaussians, thus particularly have variance $1$, is just for simplicity. Suppose that $\beta_k$ is Gaussian with mean $0$ and variance $\eta^2$. We show that the probability $p_{m}$, for $0\leq m\leq d-1$, does not depend on $\eta$. In this case, the formula for $p$ is given by \eqref{eq: formula of p} but with $\C$ being  replaced by $\eta^2 \C$. To indicate its dependence on $\eta$, we write $p_\eta$. We use a change of variable $a=\eta \tilde{a}$. Then
\begin{align*}
a^{d-1}p_\eta(a\sigma_0,\ldots, a\sigma_{d-1})\,da&=\eta^{d-1}\tilde{a}^{d-1}\frac{1}{(\sqrt{2\pi}\eta)^{d}\prod_{j=0}^{d-1} \begin{pmatrix}
d-1\\j
\end{pmatrix}}\exp\left[-\frac{\tilde{a}^2}{2}\sum_{j=0}^{d-1}\frac{\sigma_j^2}{\begin{pmatrix}
d-1\\j
\end{pmatrix}^2}\right]\eta\, d\tilde{a}
\\&=\tilde{a}^{d-1}\frac{1}{(\sqrt{2\pi})^{d}\prod_{j=0}^{d-1} \begin{pmatrix}
d-1\\j
\end{pmatrix}}\exp\left[-\frac{\tilde{a}^2}{2}\sum_{j=0}^{d-1}\frac{\sigma_j^2}{\begin{pmatrix}
d-1\\j
\end{pmatrix}^2}\right]\, d\tilde{a}
\\&=\tilde{a}^{d-1}p_1(\tilde{a}\sigma_0,\ldots, \tilde{a}\sigma_{d-1}),
\end{align*}
from which we deduce that $p_m$ does not depend on $\eta$. Similarly for the other cases, the uniform interval can be $\frac{1}{2\alpha}[-\alpha,\alpha]$ for some $\alpha>0$.
\end{remark}

In the following examples, we apply Theorem \ref{prop: pm} to obtain explicit formulas for games with small number of players ($d=3, 4$). For illustration we consider the case of normal distributions, i.e. the case C1). All formulas in these examples are executed using Mathematica. 

\subsection{Concrete examples}
\label{sec: examples}
\begin{example}[\textbf{Three-player two-strategy games: $d = 3$}]\ \\
1) One internal equilibria: $p_1=p_{1,0,1}$. We have 
\begin{align*}
& m = 1, \ k = 0, \sigma_0=1, \quad \sigma_1=x_1+x_2, \sigma_2= x_1x_2, \Delta=|x_2-x_1|
\\& q(\sigma_0,\sigma_1,\sigma_2)=\frac{1}{\left(1+x_1^2 x_2^2+\frac{1}{4} \left(x_1+x_2\right){}^2\right){}^{3/2}} |x_2-x_1|
\end{align*}
Substituting these values into \eqref{eq: pm2kG} we obtain the probability that a three-player two-strategy evolutionary game has $1$ internal equilibria 
\[
p_{1}=\frac{1}{4 \pi }\int_{\R_+}\int_{\R_-}\frac{1}{\left(1+x_1^2 x_2^2+\frac{1}{4} \left(x_1+x_2\right){}^2\right){}^{3/2}} |x_2-x_1| \,dx_1\,dx_2 = 0.5.
 \] 
2) Two internal equilibria: $p_2=p_{2,0,0}$.  We have 
\begin{align*}
& m = 2, \ k = 0, \sigma_0=1, \quad \sigma_1=x_1+x_2, \sigma_2= x_1x_2, \Delta=|x_2-x_1|
\\& q(\sigma_0,\sigma_1,\sigma_2)=\frac{1}{\left(1+x_1^2 x_2^2+\frac{1}{4} \left(x_1+x_2\right){}^2\right){}^{3/2}}|x_2-x_1|
\end{align*}

The probability that a three-player two-strategy evolutionary game has $2$ internal equilibria is
\begin{equation}
p_2=\frac{1}{8\pi}\int_{\R_+^2}\frac{1}{\left(1+x_1^2 x_2^2+\frac{1}{4} \left(x_1+x_2\right){}^2\right){}^{3/2}} |x_2-x_1| \,dx_1\,dx_2 \ \approx 0.134148.
\end{equation}
3) None-internal equilibria: the probability that a three-player two-strategy evolutionary game has none internal equilibria is $p_0=1-p_1-p_2 \ \approx 1 - 0.5 - 0.134148 =  0.365852.$
\end{example}

\begin{example}[\textbf{Four-player two-strategy games: $d=4$}]\ \\ \\
1) One internal equilibria: $p_{1}=p_{1,0,2}+p_{1,2,0}$
\\ \ \\
We first compute $p_{1,0,2}$. In this case,
\[
m=1,\quad k=0,\quad \sigma_0=1, \quad \sigma_1=x_1+x_2+x_3, \sigma_2= x_1x_2+x_1x_3+x_2x_3, \Delta=|x_2-x_1|\, |x_3-x_1|\,|x_3-x_2|.
\]
Substituting these into \eqref{eq: pm2kG} we get
\begin{multline*}
p_{1,0,2}=\frac{1}{18\pi^2}\int_{\R_-}\int_{\R_-}\int_{\R_+} \Big(1+\frac{(x_1+x_2+x_3)^2}{9}+\frac{(x_1x_2+x_1x_3+x_2x_3)^2}{9}+(x_1x_2x_3)^2\Big)^{-2} \\ \times   |x_2-x_1|\,|x_3-x_1|\, |x_3-x_2|\,dx_1\,dx_2\,dx_3 \  \approx 0.223128.
\end{multline*}
Next we compute $p_{1,2,0}$. In this case,
\begin{align*}
&m=1, \quad k=1,\quad \sigma_0=1, \sigma_1=\sigma_1(x_1,r_1e^{i\alpha_1}, r_1 e^{-i\alpha_1})=x_1+r_1e^{i\alpha_1}+ r_1 e^{-i\alpha_1}=x_1+2r_1\cos(\alpha_1),
\\&\sigma_2=\sigma_2(x_1,r_1e^{i\alpha_1}, r_1 e^{-i\alpha_1})=x_1(r_1e^{i\alpha_1}+r_1e^{-i\alpha_1})+r_1^2=2x_1r_1\cos(\alpha_1)+r_1^2,
\\& \sigma_3=\sigma_3(x_1,r_1e^{i\alpha_1}, r_1 e^{-i\alpha_1})=x_1r_1^2.
\\&\Delta=\Delta(x_1,r_1e^{i\alpha_1}, r_1 e^{-i\alpha_1})=|r_1e^{i\alpha_1}-x_1||r_1e^{-i\alpha_1}-x_1||r_1e^{i\alpha_1}-r_1e^{-i\alpha_1}|=|r_1^2-2x_1r_1\cos(\alpha_1)+x_1^2||2r_1\sin(\alpha_1)|.
\end{align*}
Substituting these into \eqref{eq: pm2kG} yields
\begin{multline*}
p_{1,2,0}=\frac{2}{9\pi^2} \int_{\R_+}\int_{[0,\pi]}\int_{\R_+}r_1\, \Big(1+\frac{(x_1+2r_1\cos(\alpha_1))^2}{9}+\frac{(2x_1r_1\cos(\alpha_1)+r_1^2)^2}{9}+(x_1r_1^2)^2\Big)^{-2}
\\ \times  |r_1^2-2x_1r_1\cos(\alpha_1)+x_1^2||2r_1\sin(\alpha_1)|\,dx_1dr_1d\alpha_1da \  \approx 0.260348.
\end{multline*}
Therefore, we obtain that
\[
p_{1}=p_{1,0,2}+p_{1,2,0}\ \approx 0.223128 +0.260348   = 0.483476. 
\]
\ \\
2) Two internal equilibria: $p_2=p_{2,0,1}$
\begin{align*}
&m=2, \quad k=0, \quad \sigma_0=1,\quad \sigma_1=x_1+x_2+x_3, \quad \sigma_2=x_1x_2+x_1x_3+x_2x_3,\quad \sigma_3=x_1x_2x_3,
\\& \Delta=|x_2-x_1|\,|x_3-x_1|\, |x_3-x_2|
\end{align*}
The probability that a four-player two-strategy evolutionary game has $2$ internal equilibria is
\begin{multline}
p_2=\frac{1}{18\pi^2}\int_{\R_+}\int_{\R_+}\int_{\R_-} \Big(1+\frac{(x_1+x_2+x_3)^2}{9}+\frac{(x_1x_2+x_1x_3+x_2x_3)^2}{9}+(x_1x_2x_3)^2\Big)^{-2} \\ \times   |x_2-x_1|\,|x_3-x_1|\, |x_3-x_2|\,dx_1\,dx_2\,dx_3   \  \approx 0.223128.
\end{multline}
3) Three internal equilibria: $p_3=p_{3,0,0}$
\begin{align*}
&m=3,\quad k=0, \quad \sigma_0=1,\quad \sigma_1=x_1+x_2+x_3, \quad \sigma_2=x_1x_2+x_1x_3+x_2x_3,\quad \sigma_3=x_1x_2x_3,
\\& \Delta=|x_2-x_1|\,|x_3-x_1|\, |x_3-x_2|
\end{align*}
The probability that a four-player two-strategy evolutionary game has $3$ internal equilibria is
\begin{align*}
p_3=\frac{1}{54\pi^2}\int_{\R_+^3} \Big(1+\frac{(x_1+x_2+x_3)^2}{9}+\frac{(x_1x_2+x_1x_3+x_2x_3)^2}{9}+(x_1x_2x_3)^2\Big)^{-2} \\ \times   |x_2-x_1|\,|x_3-x_1|\, |x_3-x_2|\,dx_1\,dx_2\,dx_3  \  \approx 0.0165236.
\end{align*}
4) None-internal equilibria: the probability that a four-player two-strategy evolutionary game has none internal equilibria is: $p_0=1-p_1-p_2-p_3  \ \approx 1 -  0.483476  - 0.223128 - 0.0165236 = 0.276872$.
\end{example}
\section{Universal estimates for $p_m$}
\label{sec: universal}
In Section~\ref{sec: distribution}, we have derived closed formulas for the probability distributions $p_m (0\leq m\leq d-1)$ of the number of internal equilibria. However, it is computationally expensive to compute these probabilities using the closed formula since it involves complex multiple-dimensional integrals. In this section, using Descartes' rule of signs and combinatorial techniques, we provide universal estimates for $p_m$. Descartes' rule of signs is a technique for determining an upper bound on the number of positive real roots of a polynomial in terms of the number of sign changes in the sequence formed by its coefficients. This rule has been applied to random polynomials before in the literature~\cite{BP32}; however this paper only obtained estimates for the expected number of zeros of a random polynomial. 

\begin{theorem}[Descartes' rule of signs, see e.g., \cite{Curtiss1918}]
Consider a polynomial of degree $n$, $p(x)=a_nx^n+\ldots+a_0$ with $a_n\neq 0$. Let $v$ be the number of variations in the sign of the coefficients $a_n,a_{n-1},\ldots,a_0$ and $n_p$ be the number of real positive zeros. Then $(v-n_p)$ is an even non-negative integer. 
\end{theorem}
We recall that an internal equilibria of a $d$-player two-strategy game is a positive root of the polynomial $P$ given in~\eqref{eq: P1}. We will apply Descartes' rule of signs to find an upper bound for the probability that a random polynomial have a certain number of positive roots. This is a problem that is of interest in its own right and may have applications elsewhere; therefore we will first study this problem for a general random polynomial of the form
\begin{equation}
\label{eq: poly}
p(y):=\sum_{k=0}^n a_k y^k,
\end{equation}
and then apply it to the polynomial $P$. It turns out that the symmetry of $\{a_k\}$ will be the key: the asymmetric case requires completely different treatment from the symmetry one.
\subsection{Estimates of $p_m$: symmetric case}
\begin{proposition} 
\label{prop: signs}
Suppose that the coefficients $a_k, 0\leq k\leq n$ in the polynomial~\eqref{eq: poly} are i.i.d and symmetrically distributed. Let $p_{k,n}, 0\leq k\leq n$ is the probability that the sequence of coefficients $(a_0,\ldots,a_{n})$ have $k$ changes of signs. Then
\begin{equation}
\label{eq: pkS}
p_{k,n}=\frac{1}{2^{n}}\begin{pmatrix}
n\\k
\end{pmatrix}.
\end{equation}
\end{proposition}
\begin{proof}
We take the sequence of coefficients $(a_0,\ldots, a_{n})$ and move from the left starting from $a_0$ to the right ending at $a_{n}$. When there is a change of sign, we write a $1$ and write a 
$0$ when there is not. Then the changes of signs form a binary sequence of length $n$. There are totally $2^{n}$ of them. Thereby $p_{k,n}$ is the probability that there are exactly $k$ numbers of $1$ in the binary sequence. There are $\begin{pmatrix}
n\\k
\end{pmatrix}$ numbers of such sequence. Since  $\{\beta_k\}$ are independent and symmetrically distributed, each sequence has a probability $\frac{1}{2^{n}}$ of occurring. From this we deduce \eqref{eq: pkS}.
\end{proof}
The next two lemmas on the sum of binomial coefficients will be used later on. 
\begin{lemma}
\label{lem: equalitiy of sum}
Let $0\leq k \leq n$ be positive integers. Then it holds that
\begin{equation}
\sum_{\substack{j=k\\ j:\text{even}}}^{n}\begin{pmatrix}
n\\
j
\end{pmatrix}=\frac{1}{2}\Bigg[\sum_{j=0}^{n-k}\begin{pmatrix}
n\\j
\end{pmatrix}+(-1)^{k}\begin{pmatrix}
n-1\\
k-1
\end{pmatrix}\Bigg]\quad \text{and}~~
\sum_{\substack{j=k\\ j:\text{odd}}}^{n}\begin{pmatrix}
n\\
j
\end{pmatrix}=\frac{1}{2}\Bigg[\sum_{j=0}^{n-k}\begin{pmatrix}
n\\j
\end{pmatrix}-(-1)^{k}\begin{pmatrix}
n-1\\
k-1
\end{pmatrix}\Bigg],
\end{equation}
where it is understood that $\begin{pmatrix}
n\\j
\end{pmatrix}=0$ if $j<0$. In particular, for $k=0$, we get
\begin{equation}
\label{eq: sum of bionomial2}
\sum_{\substack{j=0\\ j:\text{even}}}^{n}\begin{pmatrix}
n\\
j
\end{pmatrix}=\sum_{\substack{j=0\\ j:\text{odd}}}^{n}\begin{pmatrix}
n\\
j
\end{pmatrix}=2^{n-1}.
\end{equation}
\end{lemma}
\begin{proof}
Since $\sum\limits_{j=0}^{n}\begin{pmatrix}
n\\
j
\end{pmatrix}(-1)^{j} =(1+(-1))^{n}=0$, we have
\begin{equation*}
\sum_{j=k}^{n}\begin{pmatrix}
n\\
j
\end{pmatrix}(-1)^{j} =-\sum_{j=0}^{k-1}\begin{pmatrix}
n\\
j
\end{pmatrix}(-1)^{j}
\end{equation*}
According to \cite[Lemma 5.4]{DuongTran2017}
\[
\sum_{j=0}^{k-1}\begin{pmatrix}
n\\
j
\end{pmatrix}
(-1)^{j}=(-1)^{k-1}\begin{pmatrix}
n-1\\
k-1
\end{pmatrix}.
\]
Therefore, 
\begin{equation*}
\sum_{j=k}^{n}\begin{pmatrix}
n\\
j
\end{pmatrix}(-1)^{j}=(-1)^{k}\begin{pmatrix}
n-1\\
k-1
\end{pmatrix},
\end{equation*}
or equivalently:
\[
\sum_{\substack{j=k\\j:~\text{even}}}^{n}\begin{pmatrix}
n\\
j
\end{pmatrix}-\sum_{\substack{j=k\\j:~ \text{odd}}}^{n}\begin{pmatrix}
n\\
j
\end{pmatrix}
=(-1)^{k}\begin{pmatrix}
n-1\\
k-1
\end{pmatrix}.
\]
Define $\bar S_{k,n}:=\sum\limits_{j=k}^{n}\begin{pmatrix}
n\\
j
\end{pmatrix}$ and $S_{k,n}:=\sum\limits_{j=0}^{k}\begin{pmatrix}
n\\j
\end{pmatrix}$. Then using the property that $\begin{pmatrix}
n\\j
\end{pmatrix}=\begin{pmatrix}
n\\n-j
\end{pmatrix}$ we get $\bar S_{k,n}=S_{n-k,n}$ and
\begin{align*}
&\sum_{\substack{j=k\\ j:\text{even}}}^{n}\begin{pmatrix}
n\\
j
\end{pmatrix}=\frac{1}{2}\Big[\bar S_{k,n}+(-1)^{k}\begin{pmatrix}
n-1\\
k-1
\end{pmatrix}\Big]=\frac{1}{2}\Big[S_{n-k,n}+(-1)^{k}\begin{pmatrix}
n-1\\
k-1
\end{pmatrix}\Big]\quad \text{and}
\\&\sum_{\substack{j=k\\j: \text{odd}}}^{n}\begin{pmatrix}
n\\
j
\end{pmatrix}=\frac{1}{2}\Big[\bar S_{k,n}-(-1)^{k}\begin{pmatrix}
n-1\\
k-1
\end{pmatrix}\Big]=\frac{1}{2}\Big[S_{n-k,n}-(-1)^{k}\begin{pmatrix}
n-1\\
k-1
\end{pmatrix}\Big].
\end{align*}
This finishes the proof of this Lemma.
\end{proof}
The following lemma provides estimates on the sum of the first $k$ binomial coefficients.
\begin{lemma}
\label{lem: bounds of sum}
 Let $n$ and $0\leq k\leq n$ be positive integers.  We have the following estimates \cite[Lemma 8 \& Corollary 9, Chapter 10]{MacWilliamsBook} and \cite{GOTTLIEB2012}
\begin{align}
\label{eq: two-sided bounds1}
&\frac{2^{nH\big(\frac{k}{n}\big)}}{\sqrt{8k\big(1-\frac{k}{n}\big)}}\leq \sum_{j=0}^k\begin{pmatrix}
n\\j
\end{pmatrix}\leq \delta 2^{nH\big(\frac{k}{n}\big)}\quad\text{if}\quad 0\leq k\leq \frac{n}{2},\quad\text{and}
\\&2^n-\delta 2^{nH\big(\frac{k}{n}\big)}\leq\sum_{j=0}^k\begin{pmatrix}
n\\j
\end{pmatrix}\leq 2^n-\frac{2^{nH\big(\frac{k}{n}\big)}}{\sqrt{8k\big(1-\frac{k}{n}\big)}}\quad\text{if}\quad \frac{n}{2}\leq k\leq n,\label{eq: two-sided bounds2}
\end{align}
where $\delta=0.98$ and $H$ is the binary entropy function 
\begin{equation}
\label{eq: entropy function}
H(s)=-x\log_2(x)-(1-x)\log_2(1-x).
\end{equation}
In addition, if $n=2n'$ is even and $0\leq k\leq n'$, we also have the following estimate \cite[Lemma 3.8.2]{LovaszBook}
\begin{equation}
\label{eq: Lovasz bound}
\sum_{j=0}^{k-1}\begin{pmatrix}
2n'\\j
\end{pmatrix}\leq 2^{2n'-1}\begin{pmatrix}
2n'\\k
\end{pmatrix}\Big/\begin{pmatrix}
2n'\\n'
\end{pmatrix}\, .
\end{equation}
\end{lemma}
We now apply Proposition \ref{prop: signs} and Lemmas \ref{lem: equalitiy of sum}-\ref{lem: bounds of sum} to derive estimates on the probability that a $d$-player two-strategy evolutionary game has a certain number of internal equilibria. The main theorem of this section is the following.
\begin{theorem}
\label{theo: estimates}
Let $p_m, 0\leq m\leq d-1,$ be the probability that a $d$-player two-strategy has $m$ number of internal equilibria. The following assertions hold
\begin{enumerate}[(a)]
\item Upper-bound for $p_m$, for all $0\leq m\leq d-1$,
\begin{align}
p_m&\leq \frac{1}{2^{d-1}}\sum\limits_{\substack{j: j\geq m \\ j-m~\text{even}}}\begin{pmatrix}
d-1\\j
\end{pmatrix}=\frac{1}{2^d}\Bigg[ \sum_{j=0}^{d-1-m}\begin{pmatrix}
d-1\\j
\end{pmatrix}+\begin{pmatrix}
d-2\\m-1
\end{pmatrix}\Bigg]\label{eq: upper bound1}
\\&\leq \begin{cases} \frac{1}{2^d}\Bigg[
\delta 2^{(d-1)H\big(\frac{m}{d-1}\big)}+\begin{pmatrix}
d-2\\m-1
\end{pmatrix}\Bigg]\quad\text{if}~~\frac{d-1}{2}\leq m\leq d-1,\\ \\
 \frac{1}{2^d}\Bigg[2^{d-1}-\frac{2^{(d-1)H\big(\frac{m}{d-1}\big)}}{8m\big(1-\frac{m}{d-1}\big)}
+\begin{pmatrix}
d-2\\m-1
\end{pmatrix}\Bigg]\quad\text{if}~~0\leq m\leq \frac{d-1}{2}.
\end{cases}\label{eq: upper bound2}
\end{align}
As consequences, $0\leq p_m\leq \frac{1}{2}$ for all $0\leq m\leq d-1$,  $p_{d-1}\leq \frac{1}{2^{d-1}}$, $p_{d-2}\leq \frac{d-1}{2^{d-1}}$ and $\lim\limits_{d\to \infty}p_{d-1}=\lim\limits_{d\to \infty}p_{d-2}=0$.

In addition, if $d-1=2 d'$ is even and $0\leq m\leq d'$ then
\begin{equation}
\label{eq: upper Lovasz}
p_m\leq \frac{1}{2^d}\Big[2^{d-2}\begin{pmatrix}
d-1\\m-1
\end{pmatrix}\Big/\begin{pmatrix}
d-1\\d'
\end{pmatrix}+\begin{pmatrix}
d-2\\m-1
\end{pmatrix}\Big].
\end{equation}
\item Lower-bound for $p_0$ and $p_1$:
\begin{equation}
p_0\geq \frac{1}{2^{d-1}}\quad\text{and}\quad p_1\geq \frac{d-1}{2^{d-1}}.
\end{equation}
\item  For $d=2$: $p_0=p_1=\frac{1}{2}$.
\item For $d=3$: $p_1=\frac{1}{2}$.
\end{enumerate}
 \end{theorem}
\begin{proof}  
(a) This part is a combination of Decartes' rule of signs, Proposition \ref{prop: signs} and Lemmas \ref{lem: equalitiy of sum}-\ref{lem: bounds of sum}.
 In fact, as a consequence of Decartes's rule of signs and by Proposition \ref{prop: signs}, we have
\begin{equation*}
p_m\leq \sum_{\substack{j: j\geq m\\j-m:~\text{even}}}p_{j,d-1}=\frac{1}{2^{d-1}}\sum_{\substack{j: j\geq m\\j-m:~\text{even}}}\begin{pmatrix}
d-1\\j
\end{pmatrix},
\end{equation*}
which is the inequality part in \eqref{eq: upper bound1}. Next, applying Lemma \ref{lem: equalitiy of sum} for $k=m$ and $n=d-1$ and then Lemma \ref{lem: bounds of sum}, we get
\begin{align*}
\frac{1}{2^{d-1}}\sum_{\substack{k: k\geq m\\k-m:~ \text{even}}}\begin{pmatrix}
d-1\\k
\end{pmatrix}&= \begin{cases}\frac{1}{2^{d}}\Bigg[ \sum_{j=0}^{d-1-m}\begin{pmatrix}
d-1\\j
\end{pmatrix}+(-1)^m\begin{pmatrix}
d-2\\m-1
\end{pmatrix}\Bigg]\quad \text{if}~m~\text{is even} \\
\\\frac{1}{2^{d}}\Bigg[ \sum_{j=0}^{d-1-m}\begin{pmatrix}
d-1\\j
\end{pmatrix}-(-1)^m\begin{pmatrix}
d-2\\m-1
\end{pmatrix}\Bigg]\quad \text{if}~m~\text{is odd} 
\end{cases}
\\&=\frac{1}{2^d}\Bigg[ \sum_{j=0}^{d-1-m}\begin{pmatrix}
d-1\\j
\end{pmatrix}+\begin{pmatrix}
d-2\\m-1
\end{pmatrix}\Bigg]
\\&\leq\begin{cases} \frac{1}{2^d}\Bigg[
\delta 2^{(d-1)H\big(\frac{m}{d-1}\big)}+\begin{pmatrix}
d-2\\m-1
\end{pmatrix}\Bigg]\quad\text{if}~~\frac{d-1}{2}\leq m\leq d-1,\\ \\
 \frac{1}{2^d}\Bigg[2^{d-1}-\frac{2^{(d-1)H\big(\frac{m}{d-1}\big)}}{8m\big(1-\frac{m}{d-1}\big)}
+\begin{pmatrix}
d-2\\m-1
\end{pmatrix}\Bigg]\quad\text{if}~~0\leq m\leq \frac{d-1}{2}.
\end{cases}
\end{align*}
This proves the equality part in \eqref{eq: upper bound1} and \eqref{eq: upper bound2}. For the consequences: the estimate $p_m\leq\frac{1}{2}$ for all $0\leq m\leq d-1$ is followed from \eqref{eq: upper bound1} and \eqref{eq: sum of bionomial2}; the estimates $p_{d-1}\leq \frac{1}{2^{d-1}}$ and $p_{d-2}\leq \frac{d-1}{2^{d-1}}$ are special cases of \eqref{eq: upper bound1} for $m=d-1$ and $m=d-2$ respectively.

Finally, the estimate \eqref{eq: upper Lovasz} is a consequence of \eqref{eq: upper bound1} and \eqref{eq: Lovasz bound}.

\noindent (b) It follows from Decartes' rule of signs and Proposition \ref{prop: signs} that
\[
p_0\geq p_{0,d-1}=\frac{1}{2^{d-1}}\quad\text{and}~~p_{1}\geq p_{1,d-1}=\frac{d-1}{2^{d-1}}.
\]

\noindent(c) For $d=2$: from parts (a) and (b) we have
\[
\frac{1}{2}\leq p_0,p_1\leq \frac{1}{2},
\]
which implies that $p_0=p_1=\frac{1}{2}$ as claimed.

\noindent(d) Finally, for $d=3$: also from parts (a) and (b) we get
\[
\frac{1}{2}\leq p_1\leq \frac{1}{2},
\]
so $p_1=\frac{1}{2}$. We finish the proof of the Theorem.
\end{proof}
\subsection{Estimates of $p_m$: general case}
In the proof of Proposition~\ref{prop: signs} the assumption that $\{a_k\}$ are symmetrically distributed is crucial. Under this assumption, all $2^n$ binary sequences constructed there are equally distributed which results in a compact formula for $p_{k,n}$. However, when $\{a_k\}$ are not symmetrically distributed, those binary sequences are no-longer equally distributed. Thus computing $p_{k,n}$ become much more intricate. In this section, we will consider the general case where
\[
\Pp(a_i>0)=\alpha,~~\Pp(a_i<0)=1-\alpha\quad\text{ for all}~~ i=0,\ldots,n.
\]

We start with the following proposition that provides explicit formulas for $p_{k,n}$ for $k\in\{0,1,n-1,n\}$.
\begin{proposition}
\label{prop: initial values}
The following formulas hold
\begin{align*}
&\bullet\quad p_{0,n}=\alpha^{n+1}+(1-\alpha)^{n+1},\quad p_{1,n}=\begin{cases}
\frac{n}{2^n}\qquad\text{if}~\alpha=\frac{1}{2},\\
2\alpha(1-\alpha)\frac{(1-\alpha)^n-\alpha^n}{1-2\alpha}\qquad\text{if}~\alpha\neq\frac{1}{2},
\end{cases}
\\&\bullet\quad p_{n-1,n}=\begin{cases}
 n \alpha^\frac{n}{2}(1-\alpha)^\frac{n}{2}\quad\text{if $n$ even}\\
 \alpha^\frac{n+1}{2}(1-\alpha)^\frac{n+1}{2}\bigg[\frac{n+1}{2}\Big(\frac{\alpha}{1-\alpha}+\frac{1-\alpha}{\alpha}\Big)+(n-1)\bigg]\quad\text{if $n$ odd}.
\end{cases}~~\text{and}
\\&\bullet\quad p_{n,n}=\begin{cases}
\alpha^{\frac{n}{2}}(1-\alpha)^{\frac{n}{2}}\quad\text{if $n$ is even}\\
2 \alpha^\frac{n+1}{2}(1-\alpha)^\frac{n+1}{2}\quad\text{if $n$ is odd}
\end{cases}
\end{align*}
In particular, if $\alpha=\frac{1}{2}$, then $p_{0,n}=p_{1,n}=\frac{1}{2^n}\quad\text{and}\quad p_{1,n}=p_{n-1,n}=\frac{n}{2^n}$.
\end{proposition}
\begin{proof}
The four extreme cases $k\in\{0,1,n-1,n\}$ are special because we can characterise explicitly the events that the sequence $\{a_0,\ldots, a_n\}$ has $k$ changes of signs. We have
\begin{align*}
& p_{0,n}=\Pp\{a_0>0,\ldots,a_n>0\} + \Pp\{a_0<0,\ldots, a_n<0)\}
\\ &\qquad=\alpha^{n+1}+(1-\alpha)^{n+1}.
\\ 
\\&p_{1,n}=\Pp\{\cup_{k=0}^{n-1}\{a_0>0,\ldots a_k>0, a_{k+1}<0,\ldots, a_n<0\} \cup  \{a_0<0,\ldots a_k<0, a_{k+1}>0,\ldots, a_n>0\}\}
\\&\qquad=\sum_{k=0}^{n-1}\Big(\alpha^{k+1}(1-\alpha)^{n-k}+(1-\alpha)^{k+1}\alpha^{n-k}\Big)
\\&\qquad=\alpha(1-\alpha)^n\sum_{k=0}^{n-1}\Big(\frac{\alpha}{1-\alpha}\Big)^k+\alpha^n(1-\alpha)\sum_{k=0}^{n-1}\Big(\frac{1-\alpha}{\alpha}\Big)^k
\\&\qquad=\begin{cases}
\frac{n}{2^n}\qquad\text{if}~\alpha=\frac{1}{2},\\
\alpha(1-\alpha)^n\frac{1-\Big(\frac{\alpha}{1-\alpha}\Big)^n}{1-\frac{\alpha}{1-\alpha}}+\alpha^n(1-\alpha)\frac{1-\Big(\frac{1-\alpha}{\alpha}\Big)^n}{1-\frac{1-\alpha}{\alpha}}\qquad\text{if}~\alpha\neq\frac{1}{2}
\end{cases}
\\&\qquad=\begin{cases}
\frac{n}{2^n}\qquad\text{if}~\alpha=\frac{1}{2},\\
2\alpha(1-\alpha)\frac{(1-\alpha)^n-\alpha^n}{1-2\alpha}\qquad\text{if}~\alpha\neq\frac{1}{2}.
\end{cases}
\\
\\& p_{n,n}=\Pp\{\{a_0>0,a_1<0,\ldots, (-1)^n a_n>0\}\cup \{a_0<0,a_1>0,\ldots, (-1)^n a_n<0\} \}
\\&\qquad=\begin{cases}
\alpha^{\frac{n+2}{2}}(1-\alpha)^{\frac{n}{2}}+(1-\alpha)^{\frac{n+2}{2}}\alpha^{\frac{n}{2}}\quad\text{if $n$ is even}\\
2 \alpha^\frac{n+1}{2}(1-\alpha)^\frac{n+1}{2}\quad\text{if $n$ is odd}
\end{cases}
\\&\qquad=\begin{cases}
\alpha^{\frac{n}{2}}(1-\alpha)^{\frac{n}{2}}\quad\text{if $n$ is even}\\
2 \alpha^\frac{n+1}{2}(1-\alpha)^\frac{n+1}{2}\quad\text{if $n$ is odd}
\end{cases}
\end{align*}
It remains to compute $p_{n-1,n}$.
\begin{align*}
p_{n-1,n}&=\sum_{k=0}^{n-1}\Pp\{a_k~\text{and}~a_{k+1}~\text{have the same signs}~\text{and there are $n-1$ changes of signs in}~ (a_0,\ldots,a_k,a_{k+1},\ldots, a_n)\}
\\&=:\sum_{k=0}^{n-1}\gamma_{k}.
\end{align*}
We now compute $\gamma_k$. This depends on the parity of $n$ and $k$. If both $n$ and $k$ are even, then
\begin{align*}
\gamma_k&=\Pp(a_0>0,a_1<0,\ldots, a_k>0, a_{k+1}>0,\ldots a_n<0)+\Pp(a_0<0,a_1>0,\ldots, a_k<0, a_{k+1}<0,\ldots a_n>0)
\\&=(1-\alpha)^\frac{n}{2}\alpha^\frac{n+2}{2}+(1-\alpha)^{\frac{n+2}{2}}\alpha^\frac{n}{2}.
\end{align*}
If $n$ is even and $k$ is odd, then
\begin{align*}
\gamma_k&=\Pp(a_0>0,a_1<0,\ldots, a_k<0, a_{k+1}<0,\ldots a_n<0)+\Pp(a_0<0,a_1>0,\ldots, a_k>0, a_{k+1}>0,\ldots a_n>0)
\\&=\alpha^\frac{n+2}{2}(1-\alpha)^\frac{n}{2}+(1-\alpha)^\frac{n+2}{2}\alpha^\frac{n}{2}.
\end{align*}
Therefore, in both cases, i.e., if $n$ is even we get
\[
\gamma_k=\alpha^\frac{n}{2}(1-\alpha)^\frac{n}{2}.
\]
From this we deduce $p_{n-1,n}= n \alpha^\frac{n}{2}(1-\alpha)^\frac{n}{2}$.
Similarly if $n$ is odd and $k$ is even
\begin{align*}
\gamma_k&=\Pp(a_0>0,a_1<0,\ldots, a_k>0, a_{k+1}>0,\ldots a_n>0)+\Pp(a_0<0,a_1>0,\ldots, a_k<0, a_{k+1}<0,\ldots a_n<0)
\\&=(1-\alpha)^{\frac{n+3}{2}}\alpha^\frac{n-1}{2}+(1-\alpha)^\frac{n-1}{2}\alpha^\frac{n+3}{2}.
\end{align*}
If both $n$ and $k$ are odd
\begin{align*}
\gamma_k&=\Pp(a_0>0,a_1<0,\ldots, a_k<0, a_{k+1}<0,\ldots a_n>0)+\Pp(a_0<0,a_1>0,\ldots, a_k>0, a_{k+1}>0,\ldots a_n<0)
\\&=\alpha^\frac{n+1}{2}(1-\alpha)^\frac{n+1}{2}+(1-\alpha)^\frac{n+1}{2}\alpha^\frac{n+1}{2}.
\end{align*}
Then when $n$ is odd, we obtain
\begin{align*}
p_{n-1,n}&=\frac{n+1}{2}\Big[(1-\alpha)^{\frac{n+3}{2}}\alpha^\frac{n-1}{2}+(1-\alpha)^\frac{n-1}{2}\alpha^\frac{n+3}{2}\Big]+(n-1)\alpha^\frac{n+1}{2}(1-\alpha)^\frac{n+1}{2}
\\&=\alpha^\frac{n+1}{2}(1-\alpha)^\frac{n+1}{2}\bigg[\frac{n+1}{2}\Big(\frac{\alpha}{1-\alpha}+\frac{1-\alpha}{\alpha}\Big)+(n-1)\bigg].
\end{align*}
In conclusion,
\begin{equation*}
p_{n-1,n}=\begin{cases}
 n \alpha^\frac{n}{2}(1-\alpha)^\frac{n}{2}\quad\text{if $n$ even}\\
 \alpha^\frac{n+1}{2}(1-\alpha)^\frac{n+1}{2}\bigg[\frac{n+1}{2}\Big(\frac{\alpha}{1-\alpha}+\frac{1-\alpha}{\alpha}\Big)+(n-1)\bigg]\quad\text{if $n$ odd}.
\end{cases}
\end{equation*}
\end{proof}
The computations of $p_{k,n}$ for others $k$ are more involved. We will employ combinatorial techniques and derive recursive formulas for $p_{k,n}$. We define
\begin{align*}
&u_{k,n}=\Pp(\text{there are $k$ variations of signs in}~\{a_0,\ldots,a_n\}\big\vert a_{n}>0) \quad \text{and}
\\& v_{k,n}=\Pp(\text{there are $k$ variations of signs in}~\{a_0,\ldots,a_n\}\big\vert a_{n}<0).
\end{align*}
We have the following lemma.
\begin{lemma} The following recursive relations hold
\label{lem: u and v}
\begin{equation}
u_{k,n}=\alpha u_{k,n-1}+(1-\alpha)v_{k-1,n-1}\label{Heq1}
\quad \text{and}\quad v_{k,n}=\alpha u_{k-1,n-1}+(1-\alpha)v_{k,n-1}.
\end{equation}
\end{lemma}
\begin{proof}
Applying the law of total probability $P(A|B)=P(A|B,C)P(C|B)+P(A|B,\bar{C})P(\bar{C}|B)$,
we have:
\begin{align*}
&\Pp(k~\text{sign switches in} \{a_0,\ldots,a_n\}\big\vert a_{n}>0)
\\&\quad=\Pp(k~\text{ sign switches in}~\{a_0,\ldots, a_n\}\big\vert a_{n}>0, a_{n-1}>0)\Pp(a_{n-1}>0|a_{n}>0)
\\ &\qquad\quad+\Pp(k~\text{sign switches in }\{a_0,\ldots,a_n\}\big\vert a_{n}>0,a_{n-1}<0)\Pp(a_{n-1}<0|a_{n}>0).
\end{align*}
Since $a_{n-1}$ and $a_{n}$ are independent, we have $\Pp(a_{n-1}>0\big\vert a_{n}>0)=\Pp(a_{n-1}>0)$ and $P(a_{n-1}<0\big\vert a_{n}>0)=P(a_{n-1}<0)$. Therefore,
\begin{align*}
&P(k~\text{ sign switches in}~\{a_0,\ldots,a_n\}\big\vert a_{n}>0)
\\&\qquad= \Pp(k~\text{sign switches in}~\{a_0,\ldots,a_n\}\big\vert a_{n}>0,a_{n-1}>0)\Pp(a_{n-1}>0)
\\ &\qquad\qquad +\Pp(k\text{ sign switches in}~\{a_0,\ldots,a_n\}\big\vert a_{n}>0,a_{n-1}<0)P(a_{n-1}<0)
\\&\qquad=\Pp(k~\text{ sign switches in}~\{a_0,\ldots,a_{n-1}\}\big\vert a_{n-1}>0)\Pp(a_{n-1}>0)
\\& \qquad\qquad +\Pp(k-1~\text{sign switches in}~\{a_0,\ldots,a_{n-1}\}\big\vert  a_{n-1}<0)\Pp(a_{n-1}<0).
\end{align*}
Therefore we obtain the first relationship in \eqref{Heq1}. The second one is proved similarly.
\end{proof}
We can decouple the recursive relations in Lemma~\ref{lem: u and v} to obtain recursive relations for $\{u_{k,n}\}$ and $v_{k,n}$ separately. 
\begin{lemma}
\label{lem: u and v 2}
The following recursive relations hold
\begin{equation}
u_{k,n} =\alpha(1-\alpha)(u_{k-2,n-2}-u_{k,n-2})+u_{k,n-1}\quad\text{and}\quad v_{k,n}  =\alpha(1-\alpha)(v_{k-2,n-2}-v_{k,n-2})+v_{k,n-1}
\end{equation}
\end{lemma}
\begin{proof}
From \eqref{Heq1}, it follows that
\begin{equation}
v_{k-1,n-1}=\frac{u_{k,n}-\alpha u_{k,n-1}}{1-\alpha},\quad v_{k,n-1}=\frac{u_{k+1,n}-\alpha u_{k+1,n-1}}{1-\alpha}
\label{Heq1.2}
\end{equation}
Replace \eqref{Heq1.2} to \eqref{Heq1} we have
\[
\frac{u_{k+1,n+1}-\alpha u_{k+1,n}}{1-\alpha}=\alpha u_{k-1,n-1}+(1-\alpha)\frac{u_{k+1,n}-\alpha u_{k+1,n-1}}{1-\alpha}
\]
which implies that
\begin{align*}
u_{k+1,n+1}& =(1-\alpha)\alpha u_{k-1,n-1}+(1-\alpha)(u_{k+1,n}-\alpha u_{k+1,n-1})+\alpha u_{k+1,n}
\\& =(1-\alpha)\alpha u_{k-1,n-1}-\alpha(1-\alpha)u_{k+1,n-1}+u_{k+1,n}.
\end{align*}
Re-indexing we get $ u_{k,n} =(1-\alpha)\alpha(u_{k-2,n-2}-u_{k,n-2})+u_{k,n-1}$. Similarly we obtain the recursive formula for $v_{k,n}$.
\end{proof}
Using the recursive equations for $u_{k,n}$ and $v_{k,n}$ we can also derive  a recursive relation for $p_{k,n}$.
\begin{proposition}
\label{prop: recursive p}
$\{p_{k,n}\}$ satisfies the following recursive relation.
\begin{equation}
\label{eq: recursive formula p}
p_{k,n}=\alpha(1-\alpha)(p_{k-2,n-2}-p_{k,n-2})+p_{k,n-1}.
\end{equation}
\end{proposition}
\begin{proof} From Lemmas \ref{lem: u and v} and \ref{lem: u and v 2} we have
\begin{align*}
p_{k,n}&=\alpha u_{k,n}+(1-\alpha)v_{k,n}
\\&= \alpha[\alpha(1-\alpha)(u_{k-2,n-2}-u_{k,n-2})+u_{k,n-1}]
\\\qquad & +(1-\alpha)[\alpha(1-\alpha)(v_{k-2,n-2}-v_{k,n-2})+v_{k,n-1}]
\\&= \alpha(1-\alpha)[\alpha(u_{k-2,n-2}-u_{k,n-2})+(1-\alpha)(v_{k-2,n-2}-v_{k,n-2})]
\\ \qquad& +\alpha u_{k,n-1}+(1-\alpha)v_{k,n-1}
\\&= \alpha(1-\alpha)(p_{k-2,n-2}-p_{k,n-2})+p_{k,n-1}.
\end{align*}
This finishes the proof.
\end{proof}

In the next main theorem we will find explicit formulas for $p_{k,n}$ from the recursive formula in the previous lemma using the method of generating function. The case $\alpha=\frac{1}{2}$ will be a special one.
\begin{theorem}
\label{theo: explicit p}
$p_{k,n}$ is given explicitly by: for $\alpha=\frac{1}{2}$,
\[
p_{k,n}=\frac{1}{2^n}\begin{pmatrix}
n\\k
\end{pmatrix}
\]
and for $\alpha\neq \frac{1}{2}$,

If $k$ is even, $k=2k'$, then
\[
p_{k,n}=\begin{cases}
\sum_{m=\lceil\frac{n}{2}\rceil}^n \frac{n-k+1}{2m-n+1}\begin{pmatrix}
m\\
k',n-k'-m,2m-n
\end{pmatrix}(-1)^{n-k'-m}(\alpha(1-\alpha))^{n-m}~~\text{if $n$ even},\\ \\
\sum_{m=\lceil\frac{n}{2}\rceil}^n \frac{n-k+1}{2m-n+1}\begin{pmatrix}
m\\
k',n-k'-m,2m-n
\end{pmatrix}(-1)^{n-k'-m}(\alpha(1-\alpha))^{n-m}\\
\qquad+2\begin{pmatrix}
\lceil\frac{n-1}{2}\rceil \\
k'
\end{pmatrix}(-1)^{\lceil\frac{n-1}{2}\rceil-k'+1} (\alpha(1-\alpha))^{\frac{n+1}{2}}~~\text{if $n$ odd}.
\end{cases}
\]

If $k$ is odd, $k=2k'+1$, then
\[
p_{k,n}=2\,\sum_{m=\lceil\frac{n-1}{2}\rceil}^n\begin{pmatrix}
m\\
k', n-k'-m-1,2m-n+1
\end{pmatrix}(-1)^{n-k'-m-1} (\alpha(1-\alpha))^{n-m}.
\]
\end{theorem}
\begin{proof}

Set $1/A^2:=\alpha(1-\alpha)$. By Cauchy-Schartz inequality $\alpha(1-\alpha)\leq \frac{(\alpha+1-\alpha)^2}{4}=\frac{1}{4}$, it follows that $A^2\geq 4$. Define $a_{k,n}:=A^n p_{k,n}$. Substituting this relation into \eqref{eq: recursive formula p} we get the following recursive formula for $a_{k,n}$
\[
a_{k,n}=a_{k-2,n-2}-a_{k,n-2}+A a_{k,n-1}.
\]
According to Proposition \ref{prop: initial values}
\begin{align}
a_{0,n}&=A^n p_{0,n}=A^n\Big(\alpha^{n+1}+(1-\alpha)^{n+1}\Big)=\alpha\Big(\frac{\alpha}{1-\alpha}\Big)^\frac{n}{2}+(1-\alpha)\Big(\frac{1-\alpha}{\alpha}\Big)^\frac{n}{2},\label{eq: a0n}
\\ a_{1,n}&=A^n p_{1,n}=\begin{cases}
n\quad\text{if}~~\alpha=\frac{1}{2}\\
\frac{2\alpha(1-\alpha)}{1-2\alpha}\Big[\big(\frac{1-\alpha}{\alpha}\big)^\frac{n}{2}-\big(\frac{\alpha}{1-\alpha}\big)^\frac{n}{2}\Big].
\end{cases}\label{eq: a1n}
\end{align}
Also $a_{k,n}=0$ for $k>n$. Let $F(x,y)$ be the generating function of $a_{k,n}$, that is
\[
F(x,y):=\sum_{k=0}^\infty\sum_{n=0}^\infty a_{k,n}x^k y^n.
\]
Define
\[
g(x,y)=\sum_{n=0}^\infty a_{0,n} y^n+\sum_{n=0}^\infty a_{1,n} x y^n.
\]
From \eqref{eq: a0n}-\eqref{eq: a1n} we have: for $\alpha=\frac{1}{2}$
\begin{equation*}
g(x,y)=\sum_{n=0}^\infty y^n+xy\sum_{n=0}^\infty ny^{n-1}=\frac{1}{1-y}+xy\frac{d}{dy}\Big(\frac{1}{1-y}\Big)=\frac{1-y+xy}{(1-y)^2}, 
\end{equation*}
and for $\alpha\neq \frac{1}{2}$
\begin{align*}
g(x,y)&=\sum_{n=0}^\infty\bigg[\alpha\Big(\frac{\alpha}{1-\alpha}\Big)^\frac{n}{2}+(1-\alpha)\Big(\frac{1-\alpha}{\alpha}\Big)^\frac{n}{2}\bigg] y^n+\frac{2\alpha(1-\alpha)x}{1-2\alpha} \sum_{n=1}^\infty\bigg[\Big(\frac{1-\alpha}{\alpha}\Big)^\frac{n}{2}-\Big(\frac{\alpha}{1-\alpha}\Big)^\frac{n}{2}\bigg]\, y^n
\\&=\Big[\alpha-\frac{2\alpha(1-\alpha)x}{1-2\alpha}\Big]\sum_{n=0}^\infty \Big(\frac{\alpha}{1-\alpha}\Big)^\frac{n}{2}\,y^n+\Big[1-\alpha+\frac{2\alpha(1-\alpha)x}{1-2\alpha}\Big]\sum_{n=0}^\infty \Big(\frac{1-\alpha}{\alpha}\Big)^\frac{n}{2}\, y^n
\\&=\Big[\alpha-\frac{2\alpha(1-\alpha)x}{1-2\alpha}\Big]\sum_{n=0}^\infty (\alpha A)^n y^n+\Big[1-\alpha+\frac{2\alpha(1-\alpha)x}{1-2\alpha}\Big]\sum_{n=0}^\infty ((1-\alpha)A)^n y^n
\\&=\Big[\alpha-\frac{2\alpha(1-\alpha)x}{1-2\alpha}\Big]\frac{1}{1-\alpha A y}+\Big[1-\alpha+\frac{2\alpha(1-\alpha)x}{1-2\alpha}\Big]\frac{1}{1-(1-\alpha)A y}
\\&=\frac{\Big(\alpha(1-2\alpha)-2\alpha(1-\alpha)x\Big)\Big(1-(1-\alpha)Ay\Big)+\Big((1-\alpha)(1-2\alpha)+2\alpha(1-\alpha)x\Big)\Big(1-\alpha Ay\Big)}{(1-2\alpha)(1-\alpha y)(1-(1-\alpha)Ay)}
\\&=\frac{1-\frac{2y}{A}+\frac{2xy}{A}}{1-Ay+y^2}.
\end{align*}
Note that in the above computations we have the following identities
\[
\frac{1}{A^2}=\alpha(1-\alpha),\quad \frac{\alpha}{1-\alpha}=(\alpha A)^2,\quad \frac{1-\alpha}{\alpha}=(1-\alpha)^2A^2,\quad (1-\alpha Ay)(1-(1-\alpha)Ay)=1-Ay+y^2.
\]
Now we have
\begin{align}
F(x,y)&=\sum_{k=0}^\infty\sum_{n=0}^\infty a_{k,n}x^k y^n\nonumber
\\&=g(x,y)+\sum_{k=2}^\infty\sum_{n=2}^\infty (a_{k-2,n-2}-a_{k,n-2}+A a_{k,n-1})x^k y^n\nonumber
\\&=g(x,y)+\sum_{k=2}^\infty\sum_{n=2}^\infty a_{k-2,n-2} x^k y^n-\sum_{k=2}^\infty\sum_{n=2}^\infty a_{k,n-2} x^k y^n+ A \sum_{k=2}^\infty\sum_{n=2}^\infty a_{k,n-1}x^k y^n
\\&=g(x,y)+(I)+(II)+(III).\label{eq: F1}
\end{align}
We rewrite the sums (I), (II) and (III) as follow. For the first sum
\begin{align*}
(I)=\sum_{k=2}^\infty\sum_{n=2}^\infty a_{k-2,n-2} x^k y^n=x^2y^2\sum_{k=0}^\infty\sum_{n=0}^\infty a_{k,n} x^k y^n=x^2y^2 F(x,y).
\end{align*}
For the second sum
\begin{align*}
(II)=\sum_{k=2}^\infty\sum_{n=2}^\infty a_{k,n-2} x^k y^n&=\sum_{k=0}^\infty\sum_{n=2}^\infty a_{k,n-2} x^k y^n-\sum_{n=2}^\infty a_{0,n-2}y^n-\sum_{n=2}^\infty a_{1,n-2} x y^n
\\&=y^2 \sum_{k=0}^\infty\sum_{n=0}^\infty a_{k,n} x^k y^n-y^2\sum_{n=0}^\infty a_{0,n}y^n-y^2\sum_{n=1}^\infty a_{1,n} xy^n
\\&=y^2 (F(x,y)-g(x,y)).
\end{align*}
And finally for the last sum
\begin{align*}
(III)=\sum_{k=2}^\infty\sum_{n=2}^\infty a_{k,n-1} x^k y^n&=y(F(x,y)-g(x,y)).
\end{align*}
Substituting these sums back into \eqref{eq: F1} we get
\[
F(x,y)=g(x,y)+x^2y^2 F(x,y)-y^2 (F(x,y)-g(x,y))+A y(F(x,y)-g(x,y)),
\] 
which implies that
\[
F(x,y)=\frac{g(x,y)(1-Ay+y^2)}{(1-Ay+y^2-x^2y^2)}.
\]

For $\alpha=\frac{1}{2}$, we get
\begin{align*}
F(x,y)&=\frac{1-y+xy}{(1-y)^2}\frac{(1-y)^2}{(1-y)^2-x^2y^2}=\frac{1}{1-y-xy}
\\&=\sum_{n=0}^\infty (1+x)^n y^n
\\&=\sum_{n=0}^\infty \sum_{k=0}^n \begin{pmatrix}
n\\k
\end{pmatrix} x^k y^n,
\end{align*}
which implies that $\alpha_{k,n}=\begin{pmatrix}
n\\k
\end{pmatrix}$. Hence for the case $\alpha=\frac{1}{2}$, we obtain $p_{k,n}=\frac{1}{2^n}\begin{pmatrix}
n\\k
\end{pmatrix}$. 

For the case $\alpha\neq \frac{1}{2}$ we obtain
\begin{align*}
F(x,y)=\frac{1-\frac{2y}{A}+\frac{2xy}{A}}{1-Ay+y^2}\frac{1-Ay+y^2}{1-Ay+y^2-x^2y^2}=\frac{1-\frac{2y}{A}+\frac{2xy}{A}}{1-Ay+y^2-x^2y^2}.
\end{align*}
Finding the series expansion for this case is much more involved than the previous one. Using the multinomial theorem we have
\begin{align*}
\frac{1}{1-Ay+y^2-x^2y^2}&=\sum_{m=0}^\infty (x^2y^2-y^2+Ay)^m
\\&=\sum_{m=0}^\infty~\sum_{\substack{0\leq i,j,l\leq m\\i+j+l=m}}\begin{pmatrix}
m\\i,j,l
\end{pmatrix}(x^2y^2)^i(-y^2)^j(Ay)^l
\\&=\sum_{m=0}^\infty~\sum_{\substack{0\leq i,j,l\leq m\\i+j+l=m}}\begin{pmatrix}
m\\i,j,l
\end{pmatrix}(-1)^jA^l x^{2i}y^{2i+2j+l}
\\&=\sum_{m=0}^\infty~\sum_{\substack{0\leq i,l\leq m\\i+l\leq m}}\begin{pmatrix}
m\\i,m-i-l,l
\end{pmatrix}(-1)^{m-i-l}A^l x^{2i}y^{2m-l}.
\end{align*}
Therefore
\begin{align}
\label{eq: F2}
F(x,y)&=\frac{1}{A}(A-2y+2xy)\sum_{m=0}^\infty~\sum_{\substack{0\leq i,l\leq m\\i+l\leq m}}\begin{pmatrix}
m\\i,m-i-l,l
\end{pmatrix}(-1)^{m-i-l}A^l x^{2i}y^{2m-l}\nonumber
\\&=\sum_{m=0}^\infty~\sum_{\substack{0\leq i,l\leq m\\i+l\leq m}}\begin{pmatrix}
m\\i,m-i-l,l
\end{pmatrix}(-1)^{m-i-l}A^{l-1} \Big( A x^{2i}y^{2m-l}-2x^{2i}y^{2m-l+1}+2x^{2i+1}y^{2m-l+1}\Big).
\end{align}
From this we deduce that:

If $k$ is even, $k=2k'$, then to obtain the coefficient of $x^ky^n$ on the right-hand side of \eqref{eq: F2}, we select $(i,m,l)$ such that 
\[
(i=k'~\&~ 2m-l=n ~ \&~0\leq  i, l\leq m)\quad\text{or}\quad (i=k'~\&~ 2m-l+1=n~ \&~0\leq  i, l\leq m) 
\]
Then we obtain
\begin{align*}
a_{k,n}&=\sum_{m=\lceil\frac{n}{2}\rceil}^n\begin{pmatrix}
m\\
k',m-k'-(2m-n),2m-n
\end{pmatrix}(-1)^{m-k'-(2m-n)}A^{2m-n}
\\&\qquad+ 2\,\sum_{m=\lceil\frac{n-1}{2}\rceil}^n\begin{pmatrix}
m\\
k',m-k'-(2m-n+1),2m-n+1
\end{pmatrix}(-1)^{m-k'-(2m-n+1)+1}  A^{2m-n}
\\&=\sum_{m=\lceil\frac{n}{2}\rceil}^n\begin{pmatrix}
m\\
k',n-k'-m,2m-n
\end{pmatrix}(-1)^{n-k'-m}A^{2m-n}
\\&\qquad+ 2\,\sum_{m=\lceil\frac{n-1}{2}\rceil}^n\begin{pmatrix}
m\\
k', n-k'-m-1,2m-n+1
\end{pmatrix}(-1)^{n-k'-m} A^{2m-n}
\\&=\begin{cases}
\sum_{m=\lceil\frac{n}{2}\rceil}^n\Bigg[\begin{pmatrix}
m\\
k',n-k'-m,2m-n
\end{pmatrix}+2\begin{pmatrix}
m\\
k', n-k'-m-1,2m-n+1
\end{pmatrix}\Bigg](-1)^{n-k'-m}A^{2m-n}~~\text{if $n$ even}\\ \\
\sum_{m=\lceil\frac{n}{2}\rceil}^n\Bigg[\begin{pmatrix}
m\\
k',n-k'-m,2m-n
\end{pmatrix}+2\begin{pmatrix}
m\\
k', n-k'-m-1,2m-n+1
\end{pmatrix}\Bigg](-1)^{n-k'-m}A^{2m-n}\\
\qquad+2\begin{pmatrix}
\lceil\frac{n-1}{2}\rceil \\
k'
\end{pmatrix}(-1)^{\lceil\frac{n-1}{2}\rceil-k'+1} A^{-1}~~\text{if $n$ odd}
\end{cases}
\\&=\begin{cases}
\sum_{m=\lceil\frac{n}{2}\rceil}^n \frac{n-k+1}{2m-n+1}\begin{pmatrix}
m\\
k',n-k'-m,2m-n
\end{pmatrix}(-1)^{n-k'-m}A^{2m-n}~~\text{if $n$ even}\\ \\
\sum_{m=\lceil\frac{n}{2}\rceil}^n \frac{n-k+1}{2m-n+1}\begin{pmatrix}
m\\
k',n-k'-m,2m-n
\end{pmatrix}(-1)^{n-k'-m}A^{2m-n}\\
\qquad+2\begin{pmatrix}
\lceil\frac{n-1}{2}\rceil \\
k'
\end{pmatrix}(-1)^{\lceil\frac{n-1}{2}\rceil-k'+1} A^{-1}~~\text{if $n$ odd}
\end{cases}
\end{align*}
Similarly, if $k$ is odd, $k=2k'+1$, then to obtain the coefficient of $x^ky^n$ on the right-hand side of \eqref{eq: F2}, we select $(i,m,l)$ such that 
\[
(i=k'~\&~ 2m-l+1=n~ \&~0\leq  i, l\leq m) 
\]
and obtain
\begin{align*}
a_{k,n}=2\,\sum_{m=\lceil\frac{n-1}{2}\rceil}^n\begin{pmatrix}
m\\
k', n-k'-m-1,2m-n+1
\end{pmatrix}(-1)^{n-k'-m-1} A^{2m-n}.
\end{align*}
From $a_{k,n}$ we compute $p_{k,n}$ using the relations $p_{k,n}=\frac{a_{k,n}}{A^n}$ and $A^2=\frac{1}{\alpha(1-\alpha)}$ and obtain the claimed formulas. This finishes the proof of this theorem. 
\end{proof}
\begin{remark}
We can find $a_{k,n}$ by establishing a recursive relation. We have
\begin{align*}
\frac{1}{F(x,y)}&=\frac{1-Ay+y^2-x^2y^2}{1-\frac{2y}{A}+\frac{2xy}{A}}
\\&=-\frac{Axy}{2}-\frac{Ay}{2}+\frac{A^2}{4}+\frac{1-A^2/4}{1-\frac{2y}{A}+\frac{2xy}{A}}
\\&=-\frac{Axy}{2}-\frac{Ay}{2}+\frac{A^2}{4}+(1-A^2/4)\sum_{n=0}^\infty\Big(\frac{2y}{A}(1-x)\Big)^n
\\&=-\frac{Axy}{2}-\frac{Ay}{2}+\frac{A^2}{4}+(1-A^2/4)\sum_{n=0}^\infty \Big(\frac{2}{A}\Big)^n (1-x)^ny^n
\\&=-\frac{Axy}{2}-\frac{Ay}{2}+\frac{A^2}{4}+(1-A^2/4)\sum_{n=0}^\infty\sum_{k=0}^n (-1)^k C_{k,n}\Big(\frac{2}{A}\Big)^n x^ky^n
\\&=1+\big(\frac{2}{A}-A\big)y-\frac{2}{A}xy+(1-A^2/4)\sum_{n=2}^\infty\sum_{k=0}^n (-1)^k C_{k,n}\Big(\frac{2}{A}\Big)^n x^ky^n
\\&=:\sum_{n=0}^\infty\sum_{k=0}^n b_{k,n}x^k y^n:=B(x,y).
\end{align*}
where
\begin{equation}
b_{0,0}=1, \quad b_{0,1}=\frac{2}{A}-A, \quad b_{1,1}=-\frac{2}{A}\quad \text{and}\quad b_{k,n}=(1-A^2/4)(-1)^k C_{k,n}\Big(\frac{2}{A}\Big)^n ~~\text{for}~~0\leq k\leq n, n\geq 2.
\end{equation}
Using the relation that $F(x,y)B(x,y)=\Big(\sum_{n=0}^\infty \sum_{k=0}^\infty a_{k,n} x^k y^n\Big)\Big(\sum_{n'=0}^\infty \sum_{k'=0}^\infty b_{k'n'} x^{k'} y^{n'}\Big)=1$ we get the following recursive formula to determine $a_{K,N}$
\begin{align*}
a_{0,0}=\frac{1}{b_{0,0}}=1, \quad
a_{0,N}=-\sum_{n=0}^{N-1}a_{0,n}b_{0,N-n}
, \quad a_{K,N}=-\sum_{k=0}^{K-1}\sum_{n=0}^{N-1} a_{k,n}b_{K-k,N-n}.
\end{align*}
It is not trivial to obtain an explicit formula from this recursive formula. However, it is easily implemented using a computational software such as Mathematica or Mathlab.
\end{remark}
\begin{remark} Proposition \ref{prop: recursive p} provides a second-order recursive relation for the probabilities $p_{n,k}$. This relation resembles the well-known Chu-Vandermonde identity for binomial coefficient, $b_{k,n}:=\begin{pmatrix}
n\\k
\end{pmatrix}$, which is for $0<m<n$
\[
b_{k,n}=\sum\limits_{j=0}^k \begin{pmatrix}
m\\j
\end{pmatrix}b_{k-j,n-m}.
\] 
Particularly for $m=2$ we obtain
\begin{align*}
b_{k,n}&=b_{k,n-2}+2b_{k-1,n-2}+b_{k-2,n-2}
\\&=b_{k-2,n-2}-b_{k,n-2}+2(b_{k,n-2}+b_{k-1,n-2})
\\&=b_{k-2,n-2}-b_{k,n-2}+2b_{k,n-1},
\end{align*}
where the last identity is the Pascal' rule for binomial coefficients.

On the other hand, the recursive formula $p_{n,k}$ for $\alpha=\frac{1}{2}$ becomes
\[
p_{k,n}=\frac{1}{4}(p_{k-2,n-2}-p_{k,n-2})+p_{k,n-1}.
\]
Using the transformation $a_{k,n}:=\frac{1}{2^n}p_{k,n}$ as in the proof of Theorem \ref{theo: explicit p}, then 
\[
a_{k,n}=a_{k-2,n-2}-a_{k,n-2}+2a_{k,n-1},
\]
which is exactly the Chu-Vandermonde identity for $m=2$ above. Then it is no surprising that in Theorem \ref{theo: explicit p} we obtain that $a_{k,n}$ is exactly the same as the binomial coefficient $a_{k,n}=\begin{pmatrix}
n\\k
\end{pmatrix}$.
\end{remark}

\begin{example}
\begin{align*}
&\bullet\quad n=1:\quad p_{0,1}=\alpha^2+(1-\alpha)^2; \quad p_{1,1}=2\alpha(1-\alpha);
\\&\bullet\quad n=2:\quad p_{0,2}=\alpha^3+(1-\alpha)^3, \quad p_{1,2}=2\alpha(1-\alpha),\quad p_{2,2}=\alpha(1-\alpha);
\\&\bullet\quad n=3:\quad p_{0,3}=\alpha^4+(1-\alpha)^4,~~p_{1,3}=2\alpha(1-\alpha)(\alpha^2-\alpha+1),~~p_{2,3}=2\alpha(1-\alpha)(\alpha^2-\alpha+1),~~p_{3,3}=2\alpha^2(1-\alpha)^2;
\\&\bullet\quad n=4:~~p_{0,4}=\alpha^5+(1-\alpha)^5,~~p_{1,4}=2\alpha(1-\alpha)(2\alpha^2-2\alpha+1),~~p_{2,4}=3\alpha(1-\alpha)(2\alpha^2-2\alpha+1),
\\&\qquad\qquad\qquad p_{3,4}=4\alpha^2(1-\alpha)^2,~~p_{4,4}=\alpha^2(1-\alpha)^2.
\end{align*}
Direct computations verify the recursive formula for $k=2,n=4$
\[
p_{2,4}=\alpha(1-\alpha)(p_{0,2}-p_{2,2})+p_{2,3}
\]
\end{example}
We now apply Theorem~\ref{theo: explicit p} to the polynomial $P$ to obtain estimates for $p_m, 0\leq m\leq d-1$ which is the probability that a $d$-player two-strategy random evolutionary game has $m$ internal equilibria. This theorem extends Theorem~\ref{theo: estimates} for $\alpha=1/2$ to the general case although we do not get an explicit upper bound in terms of $d$ as in Theorem~\ref{theo: estimates}.
\begin{theorem}
The following assertions hold
\begin{enumerate}[(i)]
\item Upper-bound for $p_m$
\[
p_m\leq \sum_{\substack{k\geq m\\ k-m~\text{even}}} p_{k,d-1},
\]
where $p_{k,d-1}$ can be computed explicitly according to Theorem \ref{theo: explicit p} with $n$ replaced by $d-1$.
\item Lower-bound for $p_0$: $p_0\geq \alpha^{d}+(1-\alpha)^{d}\geq \frac{1}{2^{d-1}}$,
\item Lower-bound for $p_1$: $p_1\geq \begin{cases}
\frac{d-1}{2^{d-1}}\qquad\text{if}~\alpha=\frac{1}{2},\\
2\alpha(1-\alpha)\frac{(1-\alpha)^{d-1}-\alpha^{d-1}}{1-2\alpha}\qquad\text{if}~\alpha\neq\frac{1}{2},
\end{cases}$
\item Upper-bound for $p_{d-2}$: 
\begin{align*}
 p_{d-2}&\leq\begin{cases}
(d-1) \alpha^\frac{d-1}{2}(1-\alpha)^\frac{d-1}{2}\quad\text{if $d$ odd}\\
 \alpha^\frac{d}{2}(1-\alpha)^\frac{d}{2}\bigg[\frac{d}{2}\Big(\frac{\alpha}{1-\alpha}+\frac{1-\alpha}{\alpha}\Big)+(d-2)\bigg]\quad\text{if $d$ even}\end{cases}\nonumber
\\& \leq \frac{d-1}{2^{d-1}}\quad\text{when}~~d\geq 3.\label{eq: pd-2}
\end{align*}
\item Upper-bound for $p_{d-1}$: 
\begin{equation*}
q_{d-1}\leq\begin{cases}
\alpha^{\frac{d-1}{2}}(1-\alpha)^{\frac{d-1}{2}}\quad\text{if $d$ is odd}\\
2 \alpha^\frac{d}{2}(1-\alpha)^\frac{d}{2}\quad\text{if $d$ is even}
\end{cases}\leq \frac{1}{2^{d-1}}.\label{eq: pd-1}
\end{equation*}
\end{enumerate}
As consequences
\begin{enumerate}[(a)]
\item For $d=2$: $p_0=\alpha^2+(1-\alpha)^2$ and $p_1=2\alpha(1-\alpha)$.
\item For $d=3$, $p_1=2\alpha(1-\alpha)$.
\end{enumerate}
\end{theorem}
\begin{proof}
We will apply Decartes' rule of signs, Proposition \ref{prop: initial values} and \ref{theo: explicit p} for the random polynomial \eqref{eq: eqn for y}. 
It follows from Decartes' rule of signs that
\[
p_m\leq\sum_{\substack{k\geq m\\ k-m~\text{even}}} p_{k,d-1},
\]
where $p_{k,d-1}$ is given explicitly in Theorem \ref{theo: explicit p} with $n$ replaced by $d-1$. This proves the first statement. In addition, we can also deduce from Decartes' rule of signs and Proposition \ref{prop: initial values} the following estimates for special cases $m\in\{0,1,d-2,d-1\}$:
\begin{align*}
\\&\bullet~~p_0\geq p_{0,d-1}=\alpha^d+(1-\alpha)^d\geq \min_{0\leq\alpha\leq 1}[\alpha^d+(1-\alpha)^d]=\frac{1}{2^{d-1}},
\\&\bullet~~ p_1\geq p_{1,d-1}=\begin{cases}
\frac{d-1}{2^{d-1}}\qquad\text{if}~\alpha=\frac{1}{2},\\
2\alpha(1-\alpha)\frac{(1-\alpha)^{d-1}-\alpha^{d-1}}{1-2\alpha}\qquad\text{if}~\alpha\neq\frac{1}{2},
\end{cases}
\\&\bullet~~  p_{d-2}\leq p_{d-2,d-1}=\begin{cases}
(d-1) \alpha^\frac{d-1}{2}(1-\alpha)^\frac{d-1}{2}\quad\text{if $d$ odd}\\
 \alpha^\frac{d}{2}(1-\alpha)^\frac{d}{2}\bigg[\frac{d}{2}\Big(\frac{\alpha}{1-\alpha}+\frac{1-\alpha}{\alpha}\Big)+(d-2)\bigg]\quad\text{if $d$ even}.\end{cases}
 \\&\hspace*{3cm}= \begin{cases}
(d-1) (\alpha(1-\alpha))^\frac{d-1}{2}\quad\text{if $d$ odd}\\
  \frac{d}{2}(\alpha(1-\alpha))^{d/2-1}-2(\alpha(1-\alpha))^{d/2}=:\frac{d}{2}\beta^{d/2-1}-2\beta^{d/2}=:f(\beta)\quad\text{if $d$ even}
  \end{cases}
 \\&\hspace*{3cm}\leq\begin{cases}
  (d-1)(1/4)^\frac{d-1}{2}=\frac{d-1}{2^{d-1}}\quad\text{if $d$ odd}\\
  \max_{0\leq\beta\leq \frac{1}{4}} f(\beta)=\frac{d-1}{2^{d-1}}\quad\text{if $d\geq 3$ even}
  \end{cases}
  \\ &\text{where to obtain the last inequality, we have used the fact that} \\&0\leq \beta=\alpha(1-\alpha)\leq \frac{1}{4} ~\text{and}~ f'(\beta)=d\beta^{d/2-2}\Big(\frac{d}{4}-\frac{1}{2}-\beta\Big)\geq 0~~\text{when}~~0\leq\beta\leq \frac{1}{4}~~\text{and}~~ d\geq 3.
\\& \bullet~~ p_{d-1}\leq p_{d-1,d-1}=\begin{cases}
\alpha^{\frac{d-1}{2}}(1-\alpha)^{\frac{d-1}{2}}\quad\text{if $d$ is odd}\\
2 \alpha^\frac{d}{2}(1-\alpha)^\frac{d}{2}\quad\text{if $d$ is even}
\end{cases}
\\&\hspace{3cm}\leq\begin{cases}
(1/4)^\frac{d-1}{2}=\frac{1}{2^{d-1}}\quad\text{if $d$ is odd}\\
2(1/4)^\frac{d}{2}=\frac{1}{2^{d-1}} \quad\text{if $d$ is even}.
\end{cases}
 \end{align*}
This finishes the proof of the estimates $(ii)-(v)$.
\noindent For the consequences: 
for $d=2$, in this case the above estimates $(ii)-(v)$ respectively become: 
\begin{align*}
&p_0\geq \alpha^2+(1-\alpha)^2, \quad p_1\geq \begin{cases}
\frac{1}{2}\quad \text{if}~ \alpha=\frac{1}{2},\\
2\alpha (1-\alpha)\quad\text{if}~ \alpha\neq\frac{1}{2}
\end{cases}=2\alpha(1-\alpha)
\\&\text{and}\quad p_0\leq \alpha (1-\alpha)\Big[\frac{\alpha}{1-\alpha}+\frac{1-\alpha}{\alpha}\Big]=\alpha^2+(1-\alpha)^2, \quad q_1\leq 2\alpha(1-\alpha).
\end{align*}
which imply that $p_0=\alpha^2+(1-\alpha)^2, p_1=2\alpha(1-\alpha)$.

Similarly for $d=3$, estimates $(ii)$ and $(iii)$ respectively become
\[
p_1\geq \begin{cases}
\frac{1}{2}\quad\text{if}~~\alpha=\frac{1}{2},\\
2\alpha(1-\alpha)\text{if}~~\alpha\neq\frac{1}{2}
\end{cases}=2\alpha(1-\alpha),\quad\text{and}\quad p_1\leq 2\alpha(1-\alpha)
\]
from which we deduce that $p_1=2\alpha(1-\alpha)$.
\end{proof}

\section{Numerical simulations}
In this section, we perform several  numerical (sampling) simulations and  calculations  to illustrate the analytical results obtained in the previous sections.
\label{sec: numerics}
Figure \ref{fig:pm theory vs samplings} shows the values of $p_m$ for different values of $d$, for the three cases studied in Theorem \ref{prop: pm}, i.e., when $\beta_k$ are i.i.d. standard normally distributed (GD), uniformly distributed (UD1) and when $\beta_k=a_k-b_k$ with $a_k$ and $\beta_k$ being uniformly distributed (UD2). We compare results obtained from analytical formulas in Theorem \ref{prop: pm} and from samplings. The figure shows that they are in accordance with each other agreeing at least 2 digits after the decimal points. 
\begin{figure}[ht]
\centering
\includegraphics[width = \linewidth]{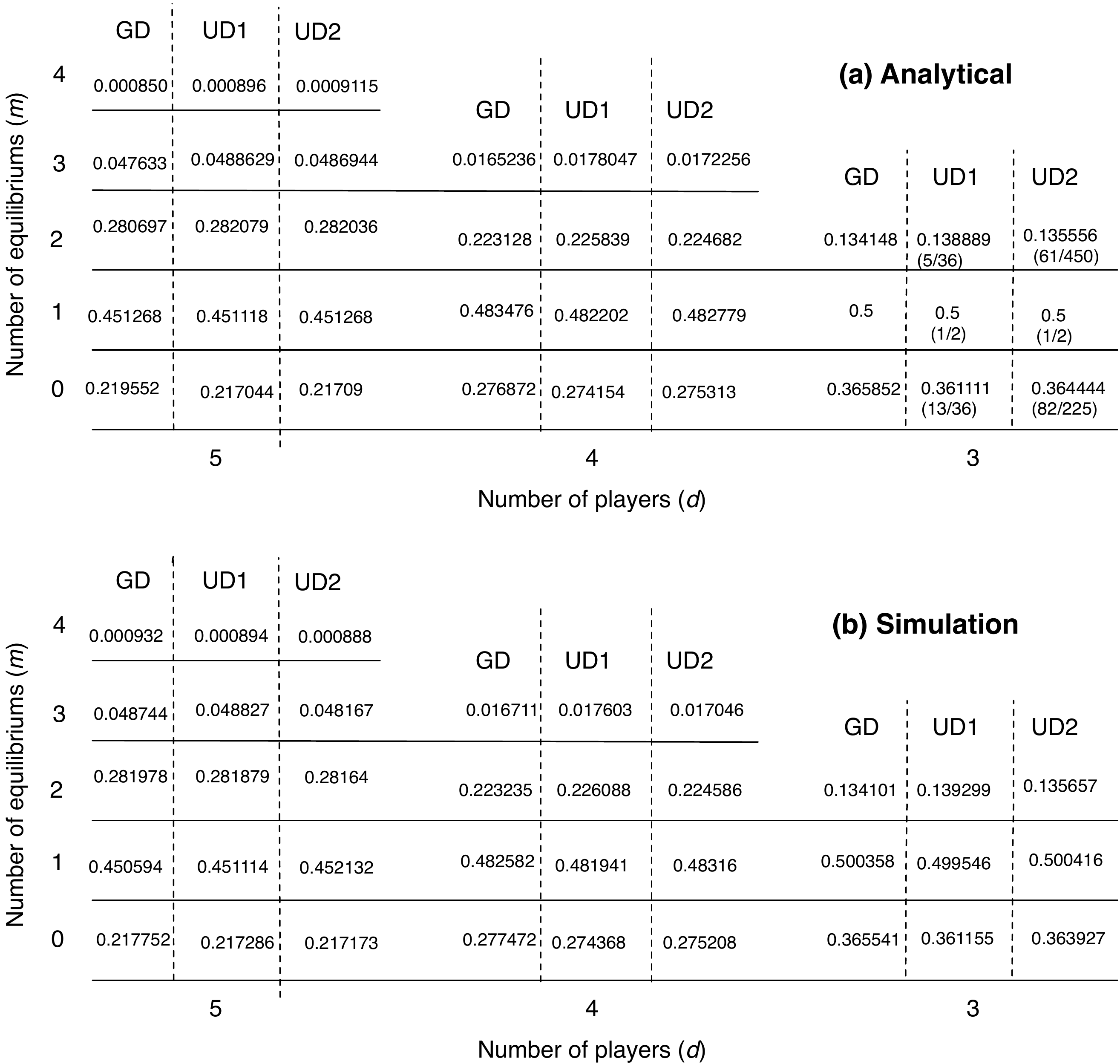}
\caption{Probability of having a certain number ($m$) of internal equilibria,  $p_m$,  for different values of $d$. 
The payoff entries $a_k$ and $b_k$ were drawn from a Gaussian distribution with variance 1 and mean 0 (GD) and from a standard uniform distribution (UD2). We also study the case where $\beta_k = a_k - b_k$ itself  is drawn from a standard uniform distribution (UD1). 
Results are obtained from analytical formulas (Theorem \ref{theo: main theo 2}) (panel a) and are based on sampling   $10^6$ payoff matrices where payoff entries are drawn from the corresponding distributions.
Analytical and simulations results are in accordance with each other. All results are obtained using Mathematica. }
\label{fig:pm theory vs samplings}
\end{figure}

\section{Further discussions and future research}
\label{sec: discussion}
In this paper, we have provided closed formulas and universal estimates for the probability distribution of the number of internal equilibria in a $d$-player two-strategy random evolutionary game. We have explored further connections between the evolutionary game theory and the random polynomial theory as discovered in our previous works~\cite{DH15,DuongHanJMB2016,DuongTranHan2017a}. We believe that the results
reported in the present work open up a new exciting avenue of research in the study of
equilibrium properties of random evolutionary games. We now provide  further discussions on these issues and possible directions for future research.

\textit{Computations of probabilities $p_m$}. Although we have found analytical formulas for $p_m$ it is computationally challenging to deal with them because of their complexity. Obtaining an effective computational method for $p_m$ would be an interesting problem for future investigation. 

\textit{Mean-field approximation theory}. Consider a general polynomial $\Pp$ as given in \eqref{eq: general polynomial} with dependent coefficients, and let $P_m([a,b],n)$ be the probability that $\Pp$ has $m$ real roots in the interval $[a,b]$ (recall that  $n$ is the degree of the polynomial, which is equal  to $d -1$ in Equation \eqref{eq: P1}). The mean-field theory in \cite{SM08} neglects the correlations between the real roots and simply consider that these roots are randomly and independently distributed on the real axis with some local
density $f(t)$ at a point $t$ where $f(t)$ being the density that can be computed from the Edelman-Kostlan theorem \cite{EK95}. Within this approximation in the large $n$ limit, the probability $P_m([a, b],n)$ is given by a non-homogeneous Poisson distribution, see \cite[Section 3.2.2 \& Equation (70)]{SM08},
\begin{equation}
\label{approximate k equilibria}
P_m([a,b],n)\approx\frac{\mu^{m}}{m!}e^{-\mu}\quad\text{where}\quad\mu=\int_a^b f(t)\,dt.
\end{equation}
Now let us apply this theory to approximate the probability $p_m$ that a random $d$-player two-strategy evolutionary game has $m$ internal equilibria. Suppose that $\beta_k$ are i.i.d. normal Gaussians, then, for large $d$, $p_m$ can be approximated by 
\begin{equation*}
p_m\approx\frac{E(d)^m}{m!}e^{-E(d)},
\end{equation*}
where $E(d)$ is the expected number of internal equilibria. According to \cite{DuongTranHan2017a}, for large $d$, we have $E(d)\sim \frac{\sqrt{2d-1}}{2}$. Therefore, we obtain the following approximate formula for $p_m$ for large $d$,
\begin{equation} 
\label{eq: approximation pm}
p_m\approx \frac{1}{m!}\Big(\frac{\sqrt{2d-1}}{2}\Big)^m \ e^{-\frac{\sqrt{2d-1}}{2}}.
\end{equation}
This formula is simple and can be easily computed; however, it is unclear to us how to quantify the errors of approximation since we do not have an efficient method to compute $p_m$ when $d$ is large, see the first point of this section. We leave this topic for future research.
\section*{Acknowledgments} 
This paper was written partly when M. H. Duong was at the Mathematics Institute, University of Warwick and was supported by ERC Starting Grant 335120.  M. H. Duong and T. A. Han acknowledge Research in Pairs Grant (No. 41606) by the London Mathematical Society to support their collaborative research. We would like to thank Dr. Dmitry Zaporozhets for his useful discussions on \cite{Zap06, GKZ2017TM}.
\bibliographystyle{alpha}
\bibliography{GTbib}
\end{document}